\documentclass[10pt]{amsart}
\usepackage{graphicx}
\usepackage{amscd}
\usepackage{amsmath}
\usepackage{amsthm}
\usepackage{amsfonts}
\usepackage{amssymb}
\usepackage{mathrsfs}
\usepackage{enumitem}
\usepackage{amsrefs}
\usepackage{xcolor}
\usepackage[colorlinks, citecolor=blue, linkcolor=red, pdfstartview=FitB]{hyperref}
\usepackage[latin1]{inputenc}

\newcommand{\bB}{{\mathbb{B}}}
\newcommand{\bC}{{\mathbb{C}}}

\newcommand{\bN}{{\mathbb{N}}}

  \newcommand{\A}{{\mathcal{A}}}
  \newcommand{\B}{{\mathcal{B}}}

  \newcommand{\E}{{\mathcal{E}}}

\renewcommand{\H}{{\mathcal{H}}}
  \newcommand{\I}{{\mathcal{I}}}
  \newcommand{\J}{{\mathcal{J}}}
  \newcommand{\K}{{\mathcal{K}}}  

  \newcommand{\M}{{\mathcal{M}}}
  \newcommand{\N}{{\mathcal{N}}}

\renewcommand{\S}{{\mathcal{S}}}
  
  \newcommand{\U}{{\mathcal{U}}}

  \newcommand{\X}{{\mathcal{X}}}
  
  \newcommand{\Z}{{\mathcal{Z}}}

\newcommand{\fA}{{\mathfrak{A}}}

\newcommand{\fC}{{\mathfrak{C}}}
\newcommand{\fc}{{\mathfrak{c}}}

\newcommand{\fd}{{\mathfrak{d}}}

\newcommand{\fH}{{\mathfrak{H}}}

\newcommand{\fT}{{\mathfrak{T}}}

\newcommand{\eps}{\varepsilon}
\renewcommand{\phi}{\varphi}

\newcommand{\upchi}{{\raise.35ex\hbox{$\chi$}}}

\newcommand{\ol}{\overline}

\newcommand{\re}{\operatorname{Re}}

\newcommand{\CB}{\operatorname{CB}}
\newcommand{\CP}{\operatorname{CP}}

\newcommand{\UCB}{\operatorname{UCB}}
\newcommand{\card}{\operatorname{card }}
\newcommand{\Pau}{\operatorname{\mathscr{P}}}

\newtheorem{lemma}{Lemma}[section]
\newtheorem{theorem}[lemma]{Theorem}

\newtheorem{corollary}[lemma]{Corollary}
\theoremstyle{definition}

\newtheorem{example}{Example}

\begin{document}
\author{Rapha\"el Clou\^atre}

\address{Department of Mathematics, University of Manitoba, Winnipeg, Manitoba, Canada R3T 2N2}

\email{raphael.clouatre@umanitoba.ca\vspace{-2ex}}
\thanks{The first author was partially supported by an NSERC Discovery Grant}
\author{Christopher Ramsey}
\email{christopher.ramsey@umanitoba.ca\vspace{-2ex}}

\begin{abstract}
We develop a completely bounded counterpart to the non-commu\-tative Choquet boundary of an operator space.
We show how the class of completely bounded linear maps is too large to accommodate our purposes. To overcome this obstacle, we isolate the subset of completely bounded linear maps on an operator space admitting a dilation of the same norm which is multiplicative on the generated $C^*$-algebra. We view such maps as analogues of the familiar unital completely contractive maps,  and we exhibit many of their structural properties. Of particular interest to us are those maps which are extremal with respect to a natural dilation order. We establish the existence of extremals and show that they have a certain unique extension property. In particular, they  give rise to $*$-homomorphisms which we use to associate to any operator space an entire scale of $C^*$-envelopes. We conjecture that these $C^*$-envelopes are all $*$-isomorphic, and verify this in some important cases.
\end{abstract}

\title[A completely bounded non-commutative Choquet boundary]{A completely bounded non-commutative Choquet boundary for operator spaces}

\maketitle

\section{Introduction}\label{S:intro}
Let $X$ be a compact metrizable space and let $\A$ be  a uniform algebra on $X$, that is a closed unital algebra of continuous functions which separates the points of $X$. The Shilov boundary is the smallest closed subset $\Sigma_\A\subset X$ with the property that the restriction map
\[
f\mapsto f|_{\Sigma_\A}, \quad f\in \A
\]
is isometric. One way of constructing $\Sigma_\A$ is to take the closure of the Choquet boundary of $\A$, which consists of the points $\xi\in X$ with the property that there is a \emph{unique} probability measure on $X$, namely the point mass at $\xi$, satisfying
\[
f(\xi)=\int_X fd\mu, \quad f\in \A.
\]
In other words, the point $\xi$ lies in the Choquet boundary of $\A$ if the point evaluation functional
\[
f\mapsto f(\xi), \quad f\in \A
\]
extends to a unique state on the $C^*$-algebra $C(X)$.

This paper will be primarily concerned with the investigation of similar questions in a non-commutative context. More precisely, we replace uniform algebras by subspaces $\M$ of arbitrary $C^*$-algebras $\fA$ such that $\fA=C^*(\M)$. Such subspaces are called operator spaces, and throughout the years the surrounding theory has developed into a powerful machine. In particular, it is now well-known that operator spaces and operator algebras can be defined in a purely abstract fashion, independent of any ambient $C^*$-algebra (these are results of Ruan \cite{ruan1988}, Choi-Effros \cite{choi1977} and Blecher-Ruan-Sinclair \cite{BRS1990} respectively). 
Despite the sophistication of this theory, it is often desirable to have access to the wealth of structure available for $C^*$-algebras, and thus the question arises of how to identify some sort of canonical smallest $C^*$-algebra containing the object of interest. Such a $C^*$-algebra would be the non-commutative analogue of the Shilov boundary of a uniform algebra.

In a seminal paper, Arveson \cite{arveson1969} initiated a program to construct the sought after $C^*$-algebra by developing a non-commutative version of the Choquet boundary. The central objects in his approach are the so-called boundary representations, certain unital completely positive linear maps having a unique extension property, much in the spirit of the defining property for points to lie in the classical Choquet boundary. Although Arveson was not able to fully realize his plan initially, the impact that his approach had is still very much felt to this day. The solution to the problem was eventually found by Hamana who constructed in \cite{hamana1979}  the $C^*$-envelope of an operator system using a different argument. Nevertheless, the objects introduced by Arveson were interesting in their own right, and spurred on significant results by Muhly-Solel \cite{muhlysolel1998} and Dritschel-McCullough \cite{dritschel2005}. Drawing from these contributions,  Arveson himself later managed to fulfill his initial vision and construct the $C^*$-envelope of a unital operator space using boundary representations, at least in the separable case \cite{arveson2008}. The separability assumption was later removed by Davidson-Kennedy in \cite{davidson2015} (see also \cite{kleski2014bdry} for related work).  There are still interesting unresolved issues regarding boundary representations, such as Arveson's hyperrigidity conjecture \cite{arveson2011} which has witnessed recent progress \cite{kleski2014hyper},\cite{DK2016}.

We mention here that the very foundation of Arveson's boundary representations approach is Stinespring's dilation theorem, which guarantees that unital completely positive maps on $C^*$-algebras can be dilated to unital $*$-homomorphisms. In particular, the construction of the $C^*$-envelope of a general (possibly non-unital) operator space presents some difficulties, although there has been a meaningful theory developed in that setting as well. Indeed, Hamana \cite{hamana1999} (see also Blecher \cite{blecher2001}) associates to any operator space its triple envelope, that is a ternary ring of operators that is the smallest such in the usual universal sense. This object can also be obtained by using a technique close in spirit to Arveson's non-commutative Choquet boundary, as was recently done by Fuller-Hartz-Lupini in \cite{FHL2016}. A device known as Paulsen's ``off-diagonal" technique and the associated generalization of Stinespring's dilation theorem \cite{paulsen1984} for completely contractive maps play a central role therein.

Heuristically, Arveson's approach yields the existence of what one may call a unital completely positive non-commutative Choquet boundary. The corresponding adaptation of Fuller-Hartz-Lupini yields the existence of a completely contractive  non-commutative Choquet boundary. However, the analogy with classical uniform algebra theory in the latter case is less accurate, as the resulting non-commutative Shilov boundary is not an algebra. We note that this apparent imperfection is somehow intrinsic due to the lack of a perfect analogue of the Stinespring dilation theorem. Accordingly, our goal in this paper is two-fold. First, we aim to develop a completely bounded counterpart to the aforementioned Choquet boundaries. Second, we wish to use this completely bounded non-commutative Choquet boundary to construct a non-commutative Shilov boundary that is still a $C^*$-algebra. To achieve these objectives, we consider operator spaces along with the extra data of a completely isometric representation on some Hilbert space. Even when we restrict our attention to the unital setting, the completely bounded theory faces the usual obstacle related to the absence of a Stinespring dilation. We overcome this difficulty by focusing on a subclass of the completely bounded linear maps. This ultimately allows us to obtain the desired objects, although subtleties arise in our construction that are not present in the works cited above.

We now outline the organization of the paper and state our main results. In Section \ref{S:prelim} we collect various preliminary notions that we require throughout. In Section \ref{S:extmachine} we adapt the machinery of \cite{dritschel2005} and \cite{davidson2015} to construct extremal elements in a very general framework. This tool is used in two different contexts later on. In Section \ref{S:CBext}, we show the following theorem (Theorem \ref{T:extmultcc}), which illustrates that the class of completely bounded linear maps is not appropriate for our purposes, and that the machinery developed in Section \ref{S:extmachine} cannot be used to produce a Shilov boundary that is an algebra by means of so-called $\CB_r$-extremal elements. 

\begin{theorem}\label{T:main5}
Let $\A\subset B(\H)$ be an operator algebra. Let $\omega:\A\to B(\H_\omega)$ be a completely bounded linear map. Assume that $\omega$ is $\CB_r(\A)$-extremal for some $r\geq 1$. Then, the following statements are equivalent.
\begin{enumerate}[label=\normalfont{(\roman*)}]
\item The map $\omega$ is multiplicative.

\item The map $\omega$ is completely contractive, and there is a $*$-homomorphism $\sigma:C^*(\A)\to B(\H_\omega)$ that agrees with $\omega$ on $\A$.

\item There is a contractive completely positive map $\Psi: C^*(\A)\to B(\H_\omega)$ that agrees with $\omega$ on $\A$.
\end{enumerate}
\end{theorem}

This theorem motivates the introduction of a subclass of maps on which we focus in subsequent sections. Roughly speaking, for every $r\geq 1$ this subclass $\Pau_r$ consists of completely bounded linear maps that have a multiplicative dilation that is well-behaved.  To make this precise, we introduce in Section \ref{S:Paulsensimprop} what we call Paulsen's similarity property, and provide a characterization of it (Theorem \ref{T:paulsenchar}) in terms of unital extensions of homomorphisms. We formally introduce the class $\Pau_r$ in Section \ref{S:Fr}, and establish many of its structural properties. We summarize some of them in the following (see Theorem \ref{T:Frdilation} and Corollary \ref{C:distance}).

\begin{theorem}\label{T:main1}
Let $\M\subset B(\H)$ be an operator space and let $\phi:\M\to B(\H_\phi)$ be an element of $ \Pau_r(\M)$, where $r\geq 1$. Then, there is a Hilbert space $\K_\phi$ containing $\H_\phi$, a $*$-homomorphism $\sigma: C^*(\M)\to B(\K_\phi)$ and an invertible operator $X\in B(\K_\phi)$ with 
\[
\|X\|=\|X^{-1}\|\leq r^{1/2}
\]
such that 
\[
\phi(a)=P_{\H_\phi} X\sigma(a)X^{-1}|_{\H_\phi},\quad a\in \M.
\]
In particular, there is a contractive completely positive linear map $\psi:C^*(\M)\to B(\H_\phi)$ with the property that
\[
\|\phi-\psi|_\M\|_{cb}\leq r-1.
\]
\end{theorem}
Section \ref{S:Frext} contains the main technical development of the paper, and is devoted to the study of the extremal elements of the class $\Pau_r$. This is where the benefits of working with this subclass are made manifest, and we obtain the following (see Corollary \ref{C:Frextnorm},  Theorem \ref{T:Frextmult} and Theorem \ref{T:Frextsimuep}).

\begin{theorem}\label{T:main2}
Let $\M\subset B(\H)$ be an operator space and let $\A_\M$ be the operator algebra that it generates. Then, the following statements hold.
\begin{enumerate}
 
  \item[\rm{(1)}] There is a $\Pau_r(\M)$-extremal element $\omega$  such that for every $n\in \bN$ and every $a\in M_n(\M)$ we have $\|a\|\leq \|\omega^{(n)}(a)\|$.
  
 \item[\rm{(2)}] Assume that $\omega$ is $\Pau_r(\M)$-extremal. Then, $\omega$ has a unique $\Pau_r$-extension to $\A_\M$, and that extension is a completely bounded homomorphism. Furthermore, there is an invertible operator $X$ with $\|X\|=\|X^{-1}\|\leq r^{1/2}$ such that the map defined as
\[
\omega_X(a)=X\omega(a)X^{-1}, \quad a\in \M
\]
has a unique contractive completely positive extension to $C^*(\M)$. That extension is a $*$-homomorphism.  
 \end{enumerate}
\end{theorem}

In Section \ref{S:C*env}, given an operator space $\M$ and a completely isometric linear map $\mu:\M\to B(\H_\mu)$, we define the $C^*$-envelope that we seek. For any $r\geq 1$ we build a $*$-homomorphism $\eps_{\mu,r}$ on $C^*(\mu(\M))$ using the extremal elements from the class $\Pau_r$. We then define the \emph{$C^*_r$-envelope} as $C^*_{e,r}(\M,\mu)=\eps_{\mu,r}(C^*(\mu(\M)))$. Although this construction is not independent of the representation $\mu$ in general, we show that it is invariant under appropriate isomorphisms, and that it has a universal property (see Theorem \ref{T:C*envinv} and Corollary \ref{C:C*envuniv}).

\begin{theorem}\label{T:main4}
Let $\M,\N$ be operator spaces and let 
\[
\mu:\M\to B(\H_\mu),\quad \nu:\N\to B(\H_\nu)
\]
be completely isometric linear maps.  Let $r\geq 1$ and let $\tau:\mu(\M)\to \nu(\N)$ be a $\Pau$-isomorphism. Then, $C^*_{e,r}(\M,\mu)$ and $C^*_{e,r}(\N,\nu)$ are unitarily equivalent. Moreover, there is a surjective $*$-homomorphism $\rho:C^*(\nu(\N))\to C^*_{e,r}((\M,\mu))$  such that $\rho  \circ  \tau=\eps_{\mu,r} $ on $\mu(\M)$. 
\end{theorem}

 We elucidate the dependence of the $C^*_r$-envelopes on the parameter $r$. We conjecture that $C_{e,r}^*(\M,\mu)=C_{e,1}^*(\M,\mu)$ for every $r\geq 1$, but are unable to prove it in general. Nevertheless, we provide supporting evidence for the conjecture in some important cases. First, we obtain the following variation on Arveson's boundary theorem \cite{arveson1972} (Theorem \ref{T:Crbdrythm}).
 
 \begin{theorem}\label{T:main6}
Let $\M$ be an operator space and let $\mu:\M\to B(\H_\mu)$ be a completely isometric linear map such that $C^*(\mu(\M))$ contains the ideal $\K$ of compact operators on $\H_\mu$. Let $r\geq 1$ and assume that there is $n\in \bN$ and $a\in M_n(\M)$ such that $\|\mu(a)+\K\|<r^{-1}\|a\|$. Then, the $C^*$-algebras $C_{e,r}^*(\M,\mu)$ and $C^*(\mu(\M))$ are $*$-isomorphic.
\end{theorem}
 
Finally, we make a connection with the original motivational example of uniform algebras (Theorem \ref{T:Crcomm}).

\begin{theorem}\label{T:main7}
Let $\A$ be a uniform algebra on a compact metrizable space $X$. Let $\Sigma_\A\subset X$ denote the Shilov boundary of $\A$. Let $\alpha:\A\to B(\H_\alpha)$ be a completely isometric algebra homomorphism.  Then, for each $r\geq 1$ the $C^*$-algebras $C^*_{e,r}(\A,\alpha)$ and $C(\Sigma_\A)$ are $*$-isomorphic.
\end{theorem}

\vspace{5mm}

\section{Preliminaries}\label{S:prelim}
\subsection{Operator spaces and completely bounded maps}\label{SS:opspaces}

We start by recalling, very briefly, terminology and basic results from operator space theory. A good reference on the subject is \cite{paulsen2002}, which the reader can consult for more details.

Let $\H$ be a Hilbert space and let $B(\H)$ denote the $C^*$-algebra of bounded linear operators on $\H$. A subspace $\M\subset B(\H)$ is called an \emph{operator space}. If $\M$ is self-adjoint and contains the identity, then it is called an \emph{operator system}. If $\M$ is a subalgebra of $B(\H)$, we say that it is an \emph{operator algebra}. As mentioned in the introduction, these three concepts can be defined abstractly, but for our purposes the previous definitions will suffice. 

Let $n\in \bN$ and let $M_n(\M)$ denote the space of $n\times n$ matrices with entries from $\M$. If we put $\H^{(n)}=\H\oplus\ldots \oplus \H$, then we may view $M_n(\M)$ as a subspace of $B(\H^{(n)})$. The norm on $M_n(\M)$ is that induced by this identification. If $\phi:\M\to B(\H_\phi)$ is a linear map, then it induces another linear map $\phi^{(n)}:M_n(\M)\to B(\H_\phi^{(n)})$ defined as
\[
\phi^{(n)}([a_{ij}]_{i,j})=[\phi(a_{ij})]_{i,j}, \quad [a_{ij}]_{i,j}\in M_n(\M).
\]
If $\M$ is an operator algebra and $\phi$ is a homomorphism, then so is $\phi^{(n)}$.\
The \emph{completely bounded norm} of $\phi$ is defined as
\[
\|\phi\|_{cb}=\sup_{n\in \bN}\|\phi^{(n)}\|.
\]
The map $\phi$ is said to be \emph{completely bounded} if $\|\phi\|_{cb}$ is finite, and \emph{completely contractive} if $\|\phi\|_{cb}\leq 1$. If  $\phi^{(n)}$ is isometric for every $n\in \bN$, then $\phi$ is said to be \emph{completely isometric}. If $\phi^{(n)}$ is positive for every $n\in \bN$, then $\phi$ is said to be \emph{completely positive}. It is well-known that if $\M$ is an operator system and $\phi$ is completely positive, then $\|\phi\|_{cb}=\|\phi(1)\|$.

Wittstock's extension theorem says that given an operator space $\M\subset B(\H)$ and a completely bounded linear map $\phi:\M\to B(\H_\phi)$, there exists a linear map $\Phi:B(\H)\to B(\H_\phi)$ that agrees with $\phi$ on $\M$ and such that $\|\Phi\|_{cb}=\|\phi\|_{cb}$. If $\M$ is an operator system and $\phi$ is completely positive, then by Arveson's extension theorem $\Phi$ can be chosen to be completely positive as well. 

Finally, we mention a very important observation of Arveson that will be used frequently without mention.  If $\M$ is a unital operator space, then $\M+\M^*$ is an operator system, and any unital completely contractive map on $\M$ extends uniquely to a unital completely positive map on $\M+\M^*$.

\subsection{Unitizations of operator spaces}\label{SS:unitization}

We typically do not assume that operator spaces contain the identity element of the ambient $B(\H)$. Even in the event that they do contain it, we typically do not assume that the maps we consider preserve the identity. Accordingly, there is a certain standard unitization procedure that we will have the occasion to use several times throughout. It associates to an operator space $\M$ a unital operator space $\Upsilon(\M)$, and to a linear map $\phi$ on $\M$ a unital linear map $\Upsilon(\phi)$ on $\Upsilon(\M)$. The notation established here will be used tacitly throughout the paper.

More precisely, let $\M\subset B(\H)$ be an operator space and let $\fA=C^*(\M)$ be the $C^*$-algebra that it generates. Then, there exist a unital $C^*$-algebra $\Upsilon(\fA)$, a unital subspace $\Upsilon(\M)\subset \Upsilon(\fA)$ and an injective $*$-homomorphism $\upsilon:\fA\to \Upsilon(\fA)$ with the property that 
\[
\Upsilon(\fA)=\upsilon(\fA)+\bC e \quad \textrm{and} \quad \Upsilon(\M)=\upsilon(\M)+\bC e
\]
where $e$ is the unit of $\Upsilon(\fA)$. To show this, we distinguish two cases. If $\fA$ is not unital, we put 
\[
\Upsilon(\fA)=\fA+\bC I_{\H}, \quad \Upsilon(\M)=\M+\bC I_{\H}
\]
and we simply let $\upsilon:\fA\to \Upsilon(\fA)$ be the inclusion. If $\fA$ is unital, we let 
\[
\Upsilon(\fA)=\fA\oplus \bC, \quad \Upsilon(\M)=\M\oplus \bC
\]
and we define the injective $*$-homomorphism  $\upsilon:\fA\to \Upsilon(\fA)$ as 
\[
\upsilon(a)=a\oplus 0, \quad a\in \fA.
\]
This establishes the claim. 

Next, we note that  the unitization can also be performed on maps. If $\phi:\M\to B(\H_\phi)$ is a linear map, then there is a unique unital linear map $\Upsilon(\phi):\Upsilon(\M)\to B(\H_\phi)$ that satisfies
\[
(\Upsilon(\phi)\circ \upsilon)(m)=\phi(m), \quad m\in \M.
\]
In the following two particular cases, more can be said about the map $\Upsilon(\phi)$.
\begin{itemize}

\item Let $\phi$ be a contractive linear map on $\M$ that extends to a contractive completely positive linear map on $\fA$. Since $\upsilon^{-1}$ is contractive and completely positive on $\upsilon(\fA)$, we may use  \cite[Proposition 2.2.1]{brownozawa2008} to see that $\Upsilon(\phi)$ is a unital map on $\Upsilon(\M)$ that extends to a completely positive map on $\Upsilon(\fA)$.

\item Let $\M$ be an operator algebra and let $\phi$ be a completely contractive homomorphism. It is easily verified that $\Upsilon(\phi)$ is a unital homomorphism, and moreover it is completely contractive by \cite[Corollary 3.3]{meyer2001} (see also \cite[Theorem 2.1.13]{blecherlemerdy2004}).

\end{itemize}
Finally, we note that given a linear map $\Phi:\Upsilon(\M)\to B(\H)$, there is a linear map $\phi:\M\to B(\H)$ such that $\Upsilon(\phi)=\Phi$ if and only if $\Phi$ is unital. Moreover, $\Phi$ is unital and completely contractive on $\Upsilon(\M)$ if and only if $\Phi=\Upsilon(\phi)$ for some linear map $\phi$ on $\M$ that extends to a contractive completely positive linear map on $\fA$.

\subsection{Ultraproducts}\label{SS:ultra}
Several of our arguments in the sequel will require the machinery of ultraproducts. The following material is folklore, but we collect it here for the convenience of the reader.

Let $\Lambda$ be a directed set and let $\U$ be a cofinal ultrafilter on $\Lambda$. For each $\lambda\in \Lambda$, let $\K_\lambda$ be a Hilbert space containing a fixed Hilbert space $\H$. We denote by $\prod_{\lambda\in \Lambda}\K_\lambda$ the Banach space of all bounded nets $(\xi_\lambda)_{\lambda\in \Lambda}$. The \emph{ultraproduct Hilbert space} $\K_\U$ of $(\K_\lambda)_{\lambda\in \Lambda}$ along $\U$ is defined as 
\[
\K_\U=\left(\prod_{\lambda\in \Lambda}\K_\lambda\right)/\Z_\U
\]
where 
\[
\Z_\U=\left\{(\xi_\lambda)_{\lambda\in \Lambda}\in \prod_{\lambda\in \Lambda}\K_\lambda:\lim_\U \|\xi_\lambda\|=0\right\}.
\]
Given $\xi=(\xi_\lambda)_{\lambda\in \Lambda}\in \prod_{\lambda\in \Lambda}\K_\lambda $, we denote its canonical image in $\K_\U$ by $[\xi]$. The inner product on $\K_\U$ is given as follows: for $\xi=(\xi_\lambda)_{\lambda\in \Lambda}\in \prod_{\lambda\in \Lambda}\K_\lambda$ and $\eta=(\eta_\lambda)_{\lambda\in \Lambda}\in \prod_{\lambda\in \Lambda}\K_\lambda$ we have
\[
\langle [\xi],[\eta]\rangle_{\K_\U}=\lim_\U \langle \xi_\lambda,\eta_\lambda \rangle_{\K_\lambda}.
\]
In particular, we see that the linear map $V:\H\to \K_\U$ defined as 
\[
Vh=[(h)_{\lambda\in \Lambda}], \quad h\in \H
\]
is an isometry.

For each $\lambda\in \Lambda$, let $T_\lambda:\K_\lambda\to \K_\lambda$ be a bounded linear operator. Assume that 
\[
\sup_{\lambda\in \Lambda} \|T_\lambda\|<\infty.
\]
Then, the \emph{ultraproduct operator} $T_\U:\K_\U\to \K_\U$ of $(T_\lambda)_{\lambda\in \Lambda}$ along $\U$ is the bounded linear operator defined as
\[
T_\U[(\xi_\lambda)_{\lambda\in \Lambda}]=[(T_\lambda \xi_\lambda)_{\lambda\in \Lambda}], \quad (\xi_\lambda)_{\lambda\in \Lambda}\in \prod_{\lambda\in \Lambda}\K_\lambda.
\]
It is readily verified that $\|T_\U\|\leq \lim_\U (\|T_\lambda\|)_{\lambda\in \Lambda}$. 

Let $\M$ be an operator space. For each $\lambda\in \Lambda$, let $\phi_\lambda: \M\to B(\K_\lambda)$ be a completely bounded linear map. Assume that
\[
\sup_{\lambda\in \Lambda}\|\phi_\lambda\|_{cb}<\infty.
\]
Then, the map 
\[
\lim_\U (\phi_\lambda)_{\lambda\in \Lambda}:\M\to B(\K_\U)
\]
defined as 
\[
\lim_\U (\phi_\lambda)_{\lambda\in \Lambda}(a)=\lim_\U (\phi_\lambda(a))_{\lambda\in \Lambda}, \quad a\in \M
\]
is linear with
\[
\left\|\lim_\U (\phi_\lambda)_{\lambda\in \Lambda} \right\|_{cb}\leq \lim_\U (\|\phi_\lambda\|_{cb})_{\lambda\in \Lambda}.
\]
If $\M$ is an operator algebra and each $\phi_\lambda$ is multiplicative, then so is $\lim_\U (\phi_\lambda)_{\lambda\in \Lambda}$.

\section{Dilation orders and maximal maps}\label{S:extmachine}

Let $\M$ be an operator space and let $r\geq 0$. For each Hilbert space $\H$ we denote by $\CB_r(\M, \H)$ the set of linear maps $\phi:\M\to B(\H)$  with $\|\phi\|_{cb}\leq r$.  Assume that for every Hilbert space $\H$, we are given a subset $\fC_r(\M,\H)$ of $\CB_r(\M,\H)$. We will say that $\fC_r(\M)$ is a \emph{subclass} of $\CB_r(\M)$ to describe such a situation. 

A partial order $\prec$ defined on the subclass $\fC_r(\M)$ is said to be a \emph{dilation order} if whenever $\phi\in \fC_r(\M,\H_\phi), \psi\in \fC_r(\M,\H_\psi)$ satisfy $\phi\prec \psi$, then we must have that $\H_\phi\subset \H_\psi$ and
\[
\|\phi(a)\xi\|\leq \|\psi(a)\xi\|, \quad \|\phi(a)^*\xi\|\leq \|\psi(a)^*\xi\|
\]
 for every $a\in \M, \xi\in\H_\phi$.
We say that the subclass $\fC_r(\M)$ has the \emph{limit property with respect to} $\prec$ if given a totally ordered set $\Lambda$ and a net
\[
\phi_\lambda\in \fC_r(\M,\H_\lambda), \quad \lambda\in \Lambda
\]
with the property that $\phi_\lambda\prec\phi_{\mu}$ if $\lambda\leq \mu$, then we can find an element $\psi\in \fC_r(\M,\ol{\cup_\lambda \H_\lambda})$  satisfying $\phi_\lambda\prec \psi$ for every $\lambda\in \Lambda$.  Moreover, we say that an element $\omega\in \fC_r(\M,\H_\omega)$ is \emph{maximal} if whenever $\delta\in \fC_r(\M,\H_\delta)$ satisfies $\omega\prec \delta$ we must have that
\[
\|\omega(a)\xi\|=\|\delta(a)\xi\|, \quad \|\omega(a)^*\xi\|=\|\delta(a)^*\xi\|
\]
for every $a\in \M, \xi\in \H_\omega$. 

We now describe a general procedure to construct maximal maps \cite{agler1988}, which we will use subsequently in two different situations. It is standard in the context of completely contractive maps on operator spaces \cite{FHL2016}, or that of unital completely positive maps on operator systems \cite{dritschel2005, davidson2015}, but because our context is different we provide the straightforward adaptations of the usual proofs.  First, we show that maximality can be achieved ``locally".

\begin{lemma}\label{L:maxlocal}
Let $\M$ be an operator space, let $r\geq 0$ and let $\fC_r(\M)$ be a subclass of $\CB_r(\M)$. Assume that $\fC_r(\M)$ has the limit property with respect to the dilation order $\prec$. Let $\phi\in \fC_r(\M,\H_\phi)$. Then, for each $a\in \M$ and $\xi\in \H_\phi$, there is $\omega\in \fC_r(\M,\H_\omega)$ such that $\phi\prec\omega$ with the property that if  $\delta\in \fC_r(\M,\H_\delta)$ and $\omega\prec\delta$ then
\[
\|\omega(a)\xi\|=\|\delta(a)\xi\|, \quad \|\omega(a)^*\xi\|=\|\delta(a)^*\xi\|.
\]

\end{lemma}
\begin{proof}
Let $\psi_0=\phi$. Then, there is $\psi_1\in \fC_r(\M,\H_1)$ such that $\psi_0\prec \psi_1$ and
\[
\|\psi_{1}(a)\xi\|\geq \sup\{ \|\psi(a)\xi\|:\psi_0\prec \psi\}-1.
\]
Next, choose $\psi_2\in \fC_r(\M,\H_2)$ such that $\psi_1\prec \psi_2$ and
\[
\|\psi_{2}(a)^*\xi\|\geq \sup\{ \|\psi(a)^*\xi\|:\psi_1\prec \psi\}-1/2.
\]
Arguing by induction, for each $n\geq 1$ we find $\psi_n\in \fC_r(\M,\H_n)$ such that $\psi_{n-1}\prec\psi_{n}$,
\[
\|\psi_{2n-1}(a)\xi\|\geq \sup\{ \|\psi(a)\xi\|:\psi_{2n-2}\prec \psi\}-1/(2n-1)
\]
and
\[
\|\psi_{2n}(a)^*\xi\|\geq \sup\{ \|\psi(a)^*\xi\|:\psi_{2n-1}\prec \psi\}-1/2n.
\]
Since $\fC_r(\M)$ has the limit property with respect to $\prec$, we find $\omega\in \fC_r(\M,\H_\omega)$ such that $\psi_n \prec \omega$ for every $n\geq 0$.
In particular, we have $\phi\prec\omega$.
Finally, given $\delta\in \fC_r(\M,\H_\delta)$ such that $\omega\prec \delta$ we have  $\psi_n\prec \delta$ for every $n\geq 1$ whence
\[
\|\psi_{2n-1}(a)\xi\|\geq \|\delta(a)\xi\|-1/(2n-1),
\]
\[
\|\psi_{2n}(a)^*\xi\|\geq \|\delta(a)^*\xi\|-1/2n
\]
and thus
\begin{align*}
\|\omega(a)\xi\|&\geq\|\psi_{2n-1}(a)\xi\|\geq \|\delta(a)\xi\|-1/(2n-1),
\end{align*}
\begin{align*}
\|\omega(a)^*\xi\|&\geq\|\psi_{2n}(a)^*\xi\|\geq \|\delta(a)^*\xi\|-1/2n
\end{align*}
for every $n\geq 1$. Hence 
\[
\|\omega(a)\xi\|\geq \|\delta(a)\xi\|, \quad \|\omega(a)^*\xi\|\geq \|\delta(a)^*\xi\|.
\]
But the reverse inequalities are always satisfied since $\prec$ is a dilation order, and we have
\[
\|\omega(a)\xi\|=\|\delta(a)\xi\|, \quad \|\omega(a)^*\xi\|= \|\delta(a)^*\xi\|
\]
as desired.
\end{proof}

Using the previous lemma, a standard induction argument yields the existence of maximal elements.

\begin{theorem}\label{T:extconstruction}
Let $\M$ be an operator space, let $r\geq 0$ and let $\fC_r(\M)$ be a subclass of $\CB_r(\M)$. Assume that $\fC_r(\M)$ has the limit property with respect to the dilation order $\prec$. Then, for each $\phi\in \fC_r(\M,\H_\phi)$  there is a maximal element $\omega\in \fC_r(\M,\H_\omega)$ such that $\phi\prec\omega$.
\end{theorem}
\begin{proof}

Let $\gamma_0$ be an ordinal with the property that there is an enumeration $\{x_\alpha:\alpha<\gamma_0\}$ of $\M\times \H_\phi$. Using transfinite recursion, we construct a net
\[
\phi_\alpha\in \fC_r(\M,\H_\alpha), \quad \alpha\leq \gamma_0
\]
such that: 
\begin{itemize}

\item $\phi\prec \phi_\alpha \prec \phi_\beta$ if $\alpha<\beta$, and

\item for every $\alpha$ we have that
if $\delta\in \fC_r(\M,\H_\delta)$ satisfies $\phi_\alpha\prec \delta$, then 
\[
\|\phi_\alpha(a)\xi\|=\|\delta(a)\xi\|, \quad \|\phi_\alpha(a)^*\xi\|=\|\delta(a)^*\xi\|
\]
for $(a,\xi)\in \{x_\beta:\beta<\alpha\}$. 
\end{itemize} 
Put $\phi_0=\phi$. Let $\alpha$ be an ordinal and assume that we have constructed $\{\phi_\beta\}_{\beta<\alpha}$. We now show how to find $\phi_\alpha$.
There are two cases to consider according to whether $\alpha$ is a successor or a limit ordinal. 

If $\alpha$ is a successor, then we let $\phi_\alpha\in \fC_r(\M,\H_\alpha)$ be the element obtained from applying Lemma \ref{L:maxlocal} to $\phi_{\alpha-1}$ and to the pair $(a,\xi)=x_{\alpha-1}$. It is readily verified that this has the required properties.

Alternatively, if $\alpha$ is a limit ordinal, then since $\fC_r(\M)$ has the limit property we find $\phi_\alpha\in\fC_r(\M,\H_\alpha)$ such that $\phi_\beta\prec \phi_\alpha$ for every ordinal $\beta<\alpha$. We claim that $\phi_\alpha$ has the desired property. In other words, we claim that if $\delta\in \fC_r(\M,\H_\delta)$ and $\phi_\alpha\prec \delta$ then
\[
\|\phi_\alpha(a)\xi\|=\|\delta(a)\xi\|, \quad \|\phi_\alpha(a)^*\xi\|=\|\delta(a)^*\xi\|
\]
for $(a,\xi)\in \{x_\beta:\beta<\alpha\}$. Fixing $\beta<\alpha$ we proceed to show that
\[
\|\phi_\alpha(a)\xi\|=\|\delta(a)\xi\|, \quad \|\phi_\alpha(a)^*\xi\|=\|\delta(a)^*\xi\|
\]
for $(a,\xi)=x_\beta$. Since $\alpha$ is a limit ordinal, there is an ordinal $\beta'$ such that $\beta<\beta'<\alpha$. In particular, we see that $\phi_\beta\prec \phi_{\beta'}\prec \phi_\alpha$ and $\phi_{\beta'}\prec \delta$ so that
\[
\|\phi_{\beta'}(a)\xi\|=\|\delta(a)\xi\|, \quad \|\phi_{\beta'}(a)^*\xi\|=\|\delta(a)^*\xi\|
\]
by the recursive assumption on $\phi_{\beta'}$. This forces 
\[
\|\phi_\alpha(a)\xi\|= \|\delta(a)\xi\|, \quad \|\phi_\alpha(a)^*\xi\|= \|\delta(a)\xi\|
\]
and establishes the existence of the collection $\{\phi_\alpha\}_{\alpha\leq \gamma_0}$. 

Next, we put $\theta_1=\phi_{\gamma_0}$ and $\X_1=\H_{\gamma_0}$.  By choice of $\gamma_0$, we see that 
\[
\|\theta_1(a)\xi\|=\|\delta(a)\xi\|, \quad  \|\theta_1(a)^*\xi\|=\|\delta(a)^*\xi\|
\]
for every $a\in \M, \xi\in \H_\phi$ and every $\delta\in \fC_r(\M,\H_\delta)$ such that $\theta_1\prec \delta$. 
By repeating the argument of the previous paragraphs and proceeding by (usual) induction, we obtain a sequence of maps
\[
\theta_n\in \fC_r(\M,\X_n), \quad n\geq 1
\]
such that 
\[
\phi\prec \theta_n\prec \theta_{n+1}
\]
 and
\[
\|\theta_n(a)\xi\|=\|\delta(a)\xi\|, \quad \|\theta_n(a)^*\xi\|=\|\delta(a)^*\xi\|, \quad  a\in \M, \xi\in \X_{n-1}
\]
whenever $\delta\in \fC_r(\M,\H_\delta)$ satisfies $\theta_n\prec\delta$. Let $\X=\ol{\cup_n \X_n}$.
Since $\fC_r(\M)$ has the limit property, we obtain a map $\omega\in \fC_r(\M,\X)$ such that
\[
\phi\prec \theta_n\prec \omega, \quad n\geq 1.
\]
Thus, we see that
\[
\|\omega(a)\xi\|=\|\theta_{n+1}(a)\xi\|=\|\delta(a)\xi\|, \quad \|\omega(a)^*\xi\|=\|\theta_{n+1}(a)^*\xi\|=\|\delta(a)^*\xi\|
\]
for every $n\in \bN$, $a\in \M$ and $\xi\in \X_n$, whenever $\delta\in \fC_r(\M,\H_\delta)$ with $\omega\prec \delta$. But we have that $\cup_n \X_n$ is dense in $\X$, so we infer 
\[
\|\omega(a)\xi\|=\|\delta(a)\xi\|, \quad \|\omega(a)^*\xi\|=\|\delta(a)^*\xi\|
\]
for every $a\in \M, \xi\in  \X$ and every $\delta\in \fC_r(\M,\H_\delta)$ such that $\omega\prec \delta$. Thus, $\omega$ is maximal.
\end{proof}

\section{Extremals in the class of completely bounded maps}\label{S:CBext}

Let $\M$ be an operator space and let $r\geq 0$.
In this section, we analyze the natural dilation order on the full class $\CB_r(\M)$. That is, given $\phi\in \CB_r(\M,\H_\phi)$ and $\psi\in \CB_r(\M,\H_\psi)$,  we write $\phi \prec \psi$ if $\H_\phi\subset \H_\psi$ and  
\[
\phi(a)=P_{\H_\phi}\psi(a)|_{\H_\phi}, \quad a\in \M.
\]
It is clear that this defines a dilation order on $\CB_r(\M)$.
We say that $\omega\in \CB_r(\M,\H_\omega)$ is $\CB_r(\M)$-\emph{extremal} if whenever $\delta\in \CB_r(\M,\H_\delta)$ satisfies $\omega\prec \delta$, then we necessarily have that $\H_\omega$  is reducing for $\delta(\M)$. 
The technical tool we need to establish the existence of $\CB_r(\M)$-extremals is the following standard fact.

\begin{lemma}\label{L:chaincc}
Let $\M$ be an operator space. Then, the class $\CB_r(\M)$ has the limit property with respect to the dilation order $\prec$.
\end{lemma}
\begin{proof}
Let $\Lambda$ be a totally ordered set. For each $\lambda\in \Lambda$,  let $\phi_\lambda\in \CB_r(\M,\H_\lambda)$. Assume that $\phi_\lambda\prec\phi_\mu$ whenever $\mu\geq \lambda$. Let $\K= \ol{\cup_{\lambda\in \Lambda} \H_\lambda}$.
For $a\in \M$ we define a linear operator 
\[
\psi(a):\cup_{\lambda\in \Lambda} \H_\lambda\to \cup_{\lambda\in \Lambda} \H_\lambda
\]
as follows. Given $x\in \cup_\lambda \H_\lambda$ and $y\in \cup_\lambda\H_\lambda$, since the spaces $\H_\lambda$ increase with $\lambda$ we may find an index $\lambda_0\in \Lambda$ such that both $x$ and $y$ lie in $\H_{\lambda_0}$. We then put
\[
\langle \psi(a) x , y\rangle=\langle \phi_{\lambda_0}(a)x ,y \rangle.
\]
We claim that $\psi(a)$ is well-defined. Indeed, assume that $x$ and $y$ both lie in $\H_\lambda\cap \H_\mu$ for some $\lambda\leq \mu$. Then, using that $\phi_\lambda\prec\phi_\mu$ we find
\[
\langle \phi_\lambda(a)x,y \rangle=\langle \phi_\mu(a)x,y \rangle
\]
which shows that $\psi(a)$ is well-defined. Moreover, it is clear that 
\[
\|\psi(a)\|\leq \sup_\lambda \|\phi_\lambda(a)\|\leq  r \|a\|
\]
so that we may extend $\psi(a)$ to a bounded linear operator on $\K$ with norm at most $r\|a\|$. 
We thus obtain a bounded linear map $\psi:\M\to B(\K)$, that is easily seen to satisfy $\|\psi\|_{cb}\leq r$. By construction, we have that 
\[
\phi_\lambda(a)=P_{\H_\lambda} \psi(a) |_{\H_\lambda}, \quad a\in \M
\]
so that $\phi_\lambda\prec \psi$.
\end{proof}

We then obtain the following useful consequence.

\begin{theorem}\label{T:extcc}
Let $\M$ be an operator space, let $r\geq 0$ and let $\phi\in \CB_r(\M,\H_\phi)$. Then, there exists a $\CB_r(\M)$-extremal $\omega:\M\to B(\H_\omega)$ with $\phi\prec \omega$.
\end{theorem}
\begin{proof}
The class $\CB_r(\M)$ has the limit property with respect to $\prec$ in view of Lemma \ref{L:chaincc}. By Theorem \ref{T:extconstruction}, we see that there is a maximal element $\omega\in \CB_r(\M,\H_\omega)$ such that $\phi\prec \omega$. The proof is now completed by noting that such an element must be $\CB_r(\M)$-extremal. Indeed, let $\delta\in \CB_r(\M,\H_\delta)$ such that $\omega\prec \delta$. We calculate for every $a\in \M$ and $\xi\in \H_\omega$ that
\begin{align*}
\|(\delta(a)-\omega(a))\xi\|^2&=\| \delta(a)\xi\|^2+\|\omega(a)\xi \|^2-2\re \langle \delta(a)\xi,\omega(a)\xi\rangle\\
&=\|\delta(a)\xi \|^2-\| \omega(a)\xi\|^2=0
\end{align*}
so that $\delta(a)\xi=\omega(a)\xi$. Likewise,
\begin{align*}
\|(\delta(a)^*-\omega(a)^*)\xi\|^2&=\| \delta(a)^*\xi\|^2+\|\omega^*(a)\xi \|^2-2\re \langle \delta(a)^*\xi,\omega(a)^*\xi\rangle\\
&=\|\delta(a)^*\xi \|^2-\| \omega(a)^*\xi\|^2=0
\end{align*}
so that $\delta(a)^*\xi=\omega(a)^*\xi$.
Hence, $\delta(\M)\H_\omega\subset \H_\omega$ and $\delta(\M)^*\H_\omega\subset \H_\omega$, so indeed $\omega$ is $\CB_r(\M)$-extremal.
\end{proof}

Before we can state and prove the main result of this section, we require some preparation. The following observation is elementary but we will need it repeatedly throughout the paper. As such, we recall the proof here for the reader's convenience.

\begin{lemma}\label{L:V}
Let $\M$ be an operator space and let $\phi:\M\to B(\H_\phi)$ be a linear map. Assume that there is a Hilbert space $\K$, two isometries 
\[
V_1:\H_\phi\to \K, \quad V_2:\H_\phi\to \K
\]
and a linear map $\psi:\M\to B(\K)$ such that
\[
\phi(a)=V_1^*\psi(a)V_2, \quad a\in \M.
\]
Then, there is a Hilbert space $\K_\phi$ containing $\H_\phi$ and two unitary operators 
\[
R_1:\K_\phi\to \K\oplus \K, \quad R_2:\K_\phi\to \K\oplus \K
\]
such that
\[
\phi(a)=P_{\H_\phi}R^*_1(\psi(a)\oplus \psi(a))R_2|_{\H_\phi}, \quad a\in \M.
\]
When $V_1=V_2$ we may choose $R_1=R_2$.
\end{lemma}
\begin{proof}

For $k=1,2$, define an isometry
\[
V_k':\H_\phi\to \K\oplus \K
\]
as
\[
V_k'\xi=V_k\xi\oplus 0, \quad \xi\in \H_\psi.
\]
Note that
\[
\phi(a)=V_1'^{*}(\psi(a)\oplus \psi(a))V_2', \quad a\in \M.
\]
Define also a unitary operator 
\[
U_k:\K \oplus \K=V_k\H_\phi\oplus (V_k\H_\phi)^\perp \oplus \K \to \H_\phi\oplus (V_k\H_\phi)^\perp \oplus \K
\]
as 
\[
U_k (V_k \xi\oplus z\oplus \eta)=\xi\oplus z\oplus \eta, \quad \xi\in \H_\phi, z\in (V_k\H_\phi)^\perp, \eta\in \K.
\]
It is readily verified that 
\[
U_k V'_k:\H_\phi\to  \H_\phi\oplus (V_k\H_\phi)^\perp\oplus \K
\]
is simply the inclusion in the first component, 
so that $(U_k V'_k)^*$ is the projection $P_{\H_\phi}$ onto the first component. For every $a\in\M$ we obtain that
\begin{align*}
\phi(a)&=V_1'^{*}(\psi(a)\oplus \psi(a))V_2'\\
&=(U_1 V_1')^* U_1 (\psi(a)\oplus \psi(a)) U_2^* (U_2 V'_2)\\
&=P_{\H_\phi} U_1(\psi(a)\oplus \psi(a))U_2^*|_{\H_\phi}.
\end{align*}
Note that $(V_1\H_\phi)^\perp \oplus \K$ and $(V_2\H_\phi)^\perp \oplus \K$ have the same dimension. Choose a unitary operator 
\[
Z:(V_1\H_\phi)^\perp\oplus \K \to (V_2 \H_\phi)^\perp \oplus \K.
\]
Then, the operator
\[
W=I\oplus Z: \H_\phi\oplus (V_1\H_\phi)^\perp\oplus \K \to \H_\phi\oplus (V_2\H_\phi)^\perp \oplus \K
\]
is unitary as well, and satisfies $P_{\H_\phi}=P_{\H_\phi}W$. Thus, we have that
\[
\phi(a)=P_{\H_\phi} WU_1(\psi(a)\oplus \psi(a))U_2^*|_{\H_\phi}, \quad a\in \M.
\]
The desired equality now follows upon setting $\K_\phi=\H_\phi\oplus (V_2\H_\phi)^\perp \oplus \K$, $R^*_1=WU_1, R_2=U^*_2$.
Finally, we note that a careful look at the proof reveals that indeed $R_1$ and $R_2$ may be chosen to be equal whenever $V_1=V_2$.
\end{proof}

Another piece of preparation we need for the main result of this section is the following generalization of Stinespring's dilation theorem.

\begin{theorem}\label{T:paulsencbdilation1}
Let $\M\subset B(\H)$ be an operator space and $\phi:\M\to B(\H_\phi)$ a completely bounded linear map. Then, there is a Hilbert space $\K$, a $*$-homomorphism $\pi:C^*(\M)\to B(\K)$, another Hilbert space $\K_\phi$ containing $\H_\phi$  and two unitary operators 
\[
R_1:\K_\phi \to \K, \quad R_2:\K_\phi\to \K
\]
with the property that 
\[
\phi(a)=\|\phi\|_{cb} P_{\H_\phi} R^*_1\pi(a)R_2|_{\H_\phi}, \quad a\in \M.
\]
\end{theorem}
\begin{proof}
The claim is trivial if $\phi=0$, so that upon replacing $\phi$ with $\phi/\|\phi\|_{cb}$, we may assume that $\phi$ is completely contractive. By \cite[Theorems 2.4 and 2.7]{paulsen1984}, 
we may find a Hilbert space $\K$, two isometries
\[
V_1:\H_\phi\to \K, \quad V_2:\H_\phi\to \K
\]
and a $*$-homomorphism $\pi:C^*(\M)\to B(\K)$ such that
\[
\phi(a)=V_1^*\pi(a)V_2, \quad a\in \M.
\]
The desired conclusion now follows at once from Lemma \ref{L:V}.
\end{proof}

We also require the following elementary non-unital adaptation of the usual Stinespring dilation theorem.

\begin{lemma}\label{L:ccp}
Let $\fA$ be a $C^*$-algebra and let $\phi:\fA\to B(\H)$ be a contractive completely positive map. Then, there is a Hilbert space $\K$ containing $\H$ and $*$-homomorphism $\sigma:\fA\to B(\K)$ such that
\[
\phi(a)=P_{\H}\sigma(a)|_{\H}, \quad a\in \fA.
\]
\end{lemma}
\begin{proof}
Consider the unital $C^*$-algebra $\Upsilon(\fA)$ and the associated unital completely positive map $\Upsilon(\phi):\Upsilon(\fA)\to B(\H)$ (see Subsection \ref{SS:unitization}). Apply the usual Stinespring dilation theorem to $\Upsilon(\phi)$ to find a Hilbert space $\K$ containing $\H$ and unital $*$-homomorphism $\pi:\Upsilon(\fA)\to B(\K)$ such that
\[
\Upsilon(\phi)(b)=P_{\H}\pi(b)|_{\H}, \quad b\in \Upsilon(\fA).
\]
Then, the map $\sigma=\pi\circ \upsilon:\fA\to B(\K)$ is a $*$-homomorphism such that 
\[
\phi(a)=P_{\H}\sigma(a)|_{\H}, \quad a\in \fA.
\]
\end{proof}

The last preliminary we need shows that $\CB_r$-extremal elements are necessarily non-degenerate.

\begin{lemma}\label{L:extnondeg}
Let $\M$ be an operator space, let $r\geq 1$ and let $\omega:\M\to B(\H_\omega)$ be $\CB_r(\M)$-extremal. Then, $\ol{\omega(\M)\H_\omega}=\H_\omega$.
\end{lemma}
\begin{proof}
Let $\H_\omega'=\ol{\omega(\M)\H_\omega}$. According to the orthogonal decomposition $\H_\omega=\H_\omega'\oplus\H'^\perp_\omega$, for every $a\in \M$ we have
\[
\omega(a)=
\begin{bmatrix}
 * & *\\
 0 & 0
\end{bmatrix}.
\]
If $\H'^\perp_\omega\neq \{0\}$, then we may choose $\phi:\M\to B(\H'^\perp_\omega)$ with $\|\phi\|_{cb}=r$. Let $\Phi:\M\to B(\H_\omega)$ be defined as
\[
\Phi(a)=\begin{bmatrix} 0 & 0\\ 0 & \phi(a) \end{bmatrix}, \quad a\in \M.
\]
Next, define $\delta:\M\to B(\H_\omega\oplus \H_\omega)$ as
\[
\delta(a)=
\begin{bmatrix}
 \omega(a) & \Phi(a) \\
 0 & 0
\end{bmatrix}, \quad a\in \M.
\]
A straightforward verification yields that $\|\delta\|_{cb}=r$. Moreover, if we identify $\H_\omega$ with $\H_\omega\oplus \{0\}\subset \H_\omega\oplus \H_\omega$, then we note that $\omega\prec \delta$, yet $\H_\omega$ is not reducing for $\delta(\M)$ and thus $\omega$ is not $\CB_r$-extremal.
\end{proof}

Finally, we come to the main result of this section that elucidates the structure of $\CB_r(\M)$-extremal elements, at least for operator algebras.

\begin{theorem}\label{T:extmultcc}
Let $\A\subset B(\H)$ be an operator algebra. Let $\omega:\A\to B(\H_\omega)$ be a completely bounded linear map. Assume that $\omega$ is $\CB_r(\A)$-extremal for some $r\geq 1$. Then, the following statements are equivalent.
\begin{enumerate}[label=\normalfont{(\roman*)}]
\item The map $\omega$ is multiplicative.

\item The map $\omega$ is completely contractive, and there is a $*$-homomorphism $\sigma:C^*(\A)\to B(\H_\omega)$ that agrees with $\omega$ on $\A$.

\item There is a contractive completely positive map $\Psi: C^*(\A)\to B(\H_\omega)$ that agrees with $\omega$ on $\A$.
\end{enumerate}
\end{theorem}
\begin{proof}
(i) $\Rightarrow$ (ii) 
Assume that $\omega$ is multiplicative. It is no loss of generality to assume that $\|\omega\|_{cb}\neq 0$. By Theorem \ref{T:paulsencbdilation1}, there is a Hilbert space $\K$, a $*$-homomorphism $\pi:C^*(\A)\to B(\K)$, another Hilbert space $\K_\omega$ containing $\H_\omega$  and two unitary operators 
\[
R_1:\K_\omega \to \K, \quad R_2:\K_\omega \to \K
\]
with the property that 
\[
\omega(a)=\|\omega\|_{cb}P_{\H_\omega} R^*_1\pi(a)R_2|_{\H_\omega}, \quad a\in \A.
\]
Using that $\omega$ is $\CB_r(\A)$-extremal we see that $R^*_1\pi(\A)R_2\H_\omega\subset\H_\omega$.  Consequently, we obtain
\[
\omega(a)=\|\omega\|_{cb}R^*_1\pi(a)R_2|_{\H_\omega}, \quad a\in \A.
\]
Now, we calculate for $a\in \A$ and $b\in \A$ that
\begin{align*}
0&=R_1(\omega(a)\omega(b)-\omega(ab)) \\
&=R_1(\|\omega\|_{cb}^2R^*_1\pi(a)R_2R^*_1\pi(b)R_2|_{\H_\omega}-\|\omega\|_{cb}R^*_1\pi(ab)R_2|_{\H_\omega})\\
&=\|\omega\|_{cb}\pi(a)(\|\omega\|_{cb}R_2R^*_1-I)\pi(b)R_2|_{\H_\omega}\\
&=\|\omega\|_{cb}\pi(a)(\|\omega\|_{cb}R_2-R_1)R^*_1\pi(b)R_2|_{\H_\omega}\\
&=\pi(a)(\|\omega\|_{cb}R_2-R_1)\omega(b).
\end{align*}
By Lemma \ref{L:extnondeg}, we see that $\H_\omega=\ol{\omega(\A)\H_\omega}$ so we find
\[
(\|\omega\|_{cb}R_2-R_1)\H_\omega\subset \cap_{a\in\A}\ker \pi(a).
\]
Hence, we have
\[
\omega(a)=\|\omega\|_{cb}R^*_1\pi(a)R_2|_{\H_\omega}=R^*_1\pi(a)R_1|_{\H_\omega}, \quad a\in \A.
\]
In particular, we see that $\omega$ is completely contractive. Using again that  $\omega$ is $\CB_r(\A)$-extremal we see that $\H_\omega$ is reducing for $R_1^*\pi(\A)R_1$ so that $R_1\H_\omega$ is invariant for $\pi(C^*(\A))$.
Thus, the map $\sigma:C^*(\A)\to B(\H_\omega)$ defined as 
\[
\sigma(a)=P_{\H_\omega}R_1^*\pi(a)R_1|_{\H_\omega}, \quad a\in C^*(\A)
\]
is a $*$-homomorphism that agrees with $\omega$ on $\A$. This establishes (ii). 

The implication (ii) $\Rightarrow$ (iii) is trivial.

(iii) $\Rightarrow$ (i) 
Assume that there is a contractive completely positive map on $\Psi: C^*(\A)\to B(\H_\omega)$ that agrees with $\omega$ on $\A$. By Lemma \ref{L:ccp}, there is a Hilbert space $\K$ containing $\H_\omega$ and a $*$-homomorphism $\pi:C^*(\A)\to B(\K)$ such that
\[
\Psi(a)=P_{\H_\omega}\pi(a)|_{\H_\omega}, \quad a\in C^*(\A).
\]
In particular, we have that
\[
\omega(a)=P_{\H_\omega}\pi(a)|_{\H_\omega}, \quad a\in \A.
\]
Since $\omega$ is $\CB_r(\A)$-extremal, we conclude that $\pi(\A)\H_\omega\subset \H_\omega$, and thus $\omega$ is multiplicative.
\end{proof}

Guided by the approach of \cite{arveson1969}, \cite{dritschel2005} and \cite{davidson2015}, one expects the $\CB_r(\M)$-extremal elements to be the natural analogues of points in the Choquet boundary. The previous theorem shows that the class of completely bounded linear maps is too large to use the associated non-commutative Choquet boundary to produce a non-commutative Shilov boundary that is an algebra. Indeed, if $r>1$, then by Theorem \ref{T:extcc} there are $\CB_r(\A)$-extremals $\omega$ such that $\|\omega\|_{cb}>1$, and those are not multiplicative by virtue of Theorem \ref{T:extmultcc}.
If one is willing to settle for weaker algebraic properties of the Shilov boundary, then some meaningful results can be obtained in the completely contractive setting in full generality, as was done in \cite{hamana1999}, \cite{blecher2001} and \cite{FHL2016}. This leads to the notion of triple envelope of an operator space. As mentioned in the introduction, we take a different path: we restrict the class of bounded linear maps under consideration in order to obtain extremals that automatically admit a multiplicative extension.

\section{Paulsen's similarity property}\label{S:Paulsensimprop}
Before we can define the subclass of completely bounded linear maps we are interested in, we make a detour to carefully examine the class of completely bounded homomorphisms. The foundation of our investigation is the following.

\begin{theorem}\label{T:paulsencbdilation2}
Let $\M\subset B(\H)$ be an operator space and $\phi:\M\to B(\H_\phi)$ a completely bounded linear map. Then, there is a Hilbert space $\K_\phi$ containing $\H_\phi$ and a completely bounded homomorphism $\theta:C^*(\M)\to B(\K_\phi)$ with the property that 
\[
\|\theta\|_{cb}\leq \left(\|\phi\|_{cb}^{1/2}+ \left(\|\phi\|_{cb}+1\right)^{1/2}\right)^2
\]
and
\[
\phi(a)=P_{\H_\phi} \theta(a)|_{\H_\phi}, \quad a\in \M.
\]
\end{theorem}
\begin{proof}
By \cite[Theorems 2.4 and 2.8]{paulsen1984}, there is a Hilbert space $\K$, an isometry $V:\H_\phi\to \K$, a $*$-homomorphism $\pi:C^*(\M)\to B(\K)$ and an invertible operator $X\in B(\K)$ such that
\[
\|X\| \|X^{-1}\| \leq \left(\|\phi\|_{cb}^{1/2}+ \left(\|\phi\|_{cb}+1\right)^{1/2}\right)^2
\]
and
\[
\phi(a)=V^* X\pi(a) X^{-1} V, \quad a\in \M.
\]
By Lemma \ref{L:V}, there is a Hilbert space $\K_\phi$ containing $\H_\phi$ and a unitary operator $U:\K\to \K_\phi$ such that if we let $\theta:C^*(\M)\to B(\K_\phi)$ be defined as
\[
\theta(a)=U X\pi(a) X^{-1}U^*, \quad a\in C^*(\M)
\]
then we find
\[
\phi(a)=P_{\H_\phi} \theta(a)|_{\H_\phi}, \quad a\in \M.
\]
Finally, note that 
\[
\|\theta\|_{cb}\leq \|X\| \|X^{-1}\|\leq  \left(\|\phi\|_{cb}^{1/2}+ \left(\|\phi\|_{cb}+1\right)^{1/2}\right)^2.
\]
\end{proof}

We also require a celebrated result of Paulsen (see \cite[Theorem 3.1]{paulsen1984} and \cite{paulsen1984PAMS}).

\begin{theorem}\label{T:paulsensim}
Let $\A$ be an operator algebra and $\theta:\A\to B(\H)$ a completely bounded homomorphism. Then, there is an invertible operator $X\in B(\H)$ with
\[
\|X\|=\|X^{-1}\|\leq \|\theta\|_{cb}^{1/2}+(\|\theta\|_{cb}+1)^{1/2}
\]
such that the map
\[
a\mapsto X\theta(a)X^{-1},\quad a\in \A
\]
is completely contractive. When $\theta$ is unital, it can be arranged that 
\[
\|X\|=\|X^{-1}\|=\|\theta\|_{cb}^{1/2}.
\]
\end{theorem}

Inspired by Theorem \ref{T:paulsensim}, we make the following definition. Let $\A$ be an operator algebra, let $r\geq 1$ and let $\theta:\A\to B(\H)$ be a completely bounded homomorphism. Then, we say that $\theta$ has \emph{Paulsen's similarity property with constant }$r$ if there is an invertible operator $X\in B(\H)$ with
\[
\|X\|=\|X^{-1}\|\leq r
\]
such that the map
\[
a\mapsto X\theta(a)X^{-1},\quad a\in \A
\]
is completely contractive. We now single out a refinement of Theorem \ref{T:paulsensim}.

\begin{corollary}\label{C:hom}
Let $\A\subset B(\H)$ be an operator algebra, let $r\geq 1$ and let $\theta:\A\to B(\H_\theta)$ be a completely bounded homomorphism that has Paulsen's similarity property with constant $r$. Then, there is a Hilbert space $\K_\theta$ containing $\H_\theta$, a $*$-homomorphism $\sigma: C^*(\A)\to B(\K_\theta)$ and an invertible operator $X\in B(\H_\theta)$ with 
\[
\|X\|=\|X^{-1}\|\leq r
\]
such that 
\[
X\theta(a)X^{-1}=P_{\H_\theta} \sigma(a)|_{\H_\theta}, \quad a\in \A.
\]
\end{corollary}
\begin{proof}
By assumption, there is an invertible operator $X\in B(\H_\theta)$ with 
\[
\|X\|=\|X^{-1}\|\leq r
\] 
such that the map $\theta_X:\A\to B(\H_\theta)$ defined as
\[
\theta_X(a)= X\theta(a)X^{-1}, \quad a\in \A
\]
is a completely contractive homomorphism. We first claim that $\theta_X$ extends to a contractive completely positive map on $C^*(\A)$.

If $\A$ is unital, then  $P=\theta_X(I_{\H})$ is a contractive idempotent, and thus a self-adjoint projection in $\theta_X(\A)'$. Hence, 
\[
\theta_X(a)=P\theta_X(a)P\oplus 0, \quad a\in \A
\]
and the \emph{unital} completely contractive map
\[
a\mapsto P\theta_X(a)P\in B(P\H_\theta), \quad a\in \A
\]
extends to a contractive completely positive map on $C^*(\A)$, so the same is true of $\theta_X$. In the non-unital case, we saw in Subsection \ref{SS:unitization} that there is unital completely contractive homomorphism  on $\A+\bC I_{\H}$ extending $\theta_X$. We conclude  that $\theta_X$ extends to a unital completely positive map on $C^*(\A+\bC I_{\H})$. The claim is established. 

By Lemma \ref{L:ccp}, there is a Hilbert space $\K_\theta$ containing $\H_\theta$ and a $*$-homomorphism $\sigma:C^*(\A)\to B(\K_\theta)$ such that
\[
X\theta(a)X^{-1}=P_{\H_\theta}\sigma(a)|_{\H_\theta}, \quad a\in \A
\]
as desired.
\end{proof}

We examine Paulsen's similarity property more closely. First observe that completely contractive homomorphisms trivially have Paulsen's similarity property with constant $1$. Moreover, Theorem \ref{T:paulsensim} shows that any completely bounded homomorphism $\theta$ has Paulsen's similarity property with constant $r$ provided that $r\geq \|\theta\|_{cb}^{1/2}+(\|\theta\|_{cb}+1)^{1/2}$. In the other direction, if  $\theta$ has Paulsen's similarity property with constant $r$, then clearly $r\geq  \|\theta\|_{cb}^{1/2}$. Any unital completely bounded homomorphism $\theta$ has Paulsen's similarity property with the sharp constant $\|\theta\|_{cb}^{1/2}$, by Theorem \ref{T:paulsensim} again.

The next result is a characterization of Paulsen's similarity property in terms of the existence of unital completely bounded extensions. It uses the unitization procedure described in Subsection \ref{SS:unitization}.

\begin{theorem}\label{T:paulsenchar}
Let $\A$ be an operator algebra, let $r\geq 1$ and let $\theta:\A\to B(\H)$ be a homomorphism. Then, the following statements are equivalent.
\begin{enumerate}

\item[\rm{(i)}] The map $\theta$ has Paulsen's similarity property with constant $r$.

\item[\rm{(ii)}] The unital homomorphism $\Upsilon(\theta):\Upsilon(\A)\to B(\H)$ satisfies $\|\Upsilon(\theta)\|_{cb}\leq r^2 $.

\end{enumerate}
\end{theorem}
\begin{proof}
(i) $\Rightarrow$ (ii) 
Assume that $\theta$ has Paulsen's similarity property with constant $r$, so there is an invertible operator $X\in B(\H)$ with
\[
\|X\|=\|X^{-1}\|\leq r
\]
such that the map $\theta_X:\A\to B(\H)$ defined as
\[
a\mapsto X\theta(a)X^{-1},\quad a\in \A
\]
is completely contractive. Thus, the unital homomorphism $\Upsilon(\theta_X)$ is completely contractive as seen in Subsection \ref{SS:unitization}. It is readily verified that
\[
\Upsilon(\theta_X)(b)=X(\Upsilon(\theta)(b))X^{-1}, \quad b\in \Upsilon(\A)
\]
so that
\[
\|\Upsilon(\theta)\|_{cb}\leq \|X\| \|X^{-1}\| \|\Upsilon(\theta_X)\|_{cb}\leq r^2.
\]

(ii) $\Rightarrow$ (i)
Conversely, assume that the unital homomorphism $\Upsilon(\theta):\Upsilon(\A)\to B(\H)$ satisfies $\|\Upsilon(\theta)\|_{cb}\leq r^2 $. By Theorem \ref{T:paulsensim}, there is an invertible operator $X\in B(\H)$ with
\[
\|X\|=\|X^{-1}\|=\|\Upsilon(\theta)\|_{cb}^{1/2}\leq r
\]
such that the map
\[
b\mapsto X\Upsilon(\theta)(b)X^{-1},\quad b\in \Upsilon(\A)
\]
is completely contractive. In particular, we see that the map
\[
a\mapsto X\Upsilon(\theta)(\upsilon(a))X^{-1}=X\theta(a)X^{-1}, \quad a\in \A
\]
is completely contractive, since $\upsilon:\A\to \Upsilon(\A)$ is completely isometric. We conclude that $\theta$ has Paulsen's similarity property with constant $r$.
\end{proof}
A useful consequence is the following, that describes how Paulsen's similarity property behaves with respect to direct sums.

\begin{corollary}\label{C:paulsendirectsum}
Let $\A$ be an operator algebra and let $\theta_1: \A\to B(\H_1), \theta_2 :\A\to B(\H_2)$ be completely bounded homomorphisms. Let $\theta=\theta_1\oplus \theta_2:\A\to B(\H_1\oplus \H_2)$. Then, the following statements hold.
\begin{enumerate}

\item[\rm{(1)}] If $\theta_1$ and $\theta_2$ have  Paulsen's similarity property with constants $r_1$ and $r_2$ respectively, then $\theta$ has Paulsen's similarity property with constant $\max\{r_1,r_2\}$.

\item[\rm{(2)}] If $\theta$ has  Paulsen's similarity property with constant $r$,  then $\theta_1$ and $\theta_2$ both have Paulsen's similarity property with constant $r$.
\end{enumerate}
\end{corollary}
\begin{proof}
(1) Assume that $\theta_1$ and $\theta_2$ have Paulsen's similarity property with constants $r_1$ and $r_2$ respectively. Accordingly, we can find invertible operators $X_1\in B(\H_1),X_2\in B(\H_2)$  with
\[
\|X_1\|=\|X_1^{-1}\|\leq r_1 ,  \quad \|X_2\|=\|X_2^{-1}\|\leq r_2
\]
such that the maps
\[
a\mapsto X_1 \theta_1(a) X_1^{-1}, \quad a\in \A
\]
and
\[
a\mapsto X_2 \theta_2(a) X_2^{-1}, \quad a\in \A
\]
are completely contractive. If we put $X=X_1\oplus X_2$, then
\[
\|X\|=\|X^{-1}\|\leq \max\{r_1,r_2 \},
\]
and the map
\[
a\mapsto X (\theta_1(a)\oplus  \theta_2(a)) X^{-1}, \quad a\in \A
\]
is completely contractive, so that $\theta_1\oplus \theta_2$ has Paulsen's similarity property with constant $\max\{r_1,r_2\}$.

(2) Assume that $\theta=\theta_1\oplus \theta_2$ has Paulsen's similarity property with constant $r$. It suffices to establish the claim for $\theta_1$. Using Theorem \ref{T:paulsenchar}, we see that $\|\Upsilon(\theta)\|_{cb}\leq r^2 $. By construction of $\Upsilon(\theta)$, we see that $\H_1$ and $\H_2$ are reducing for $\Upsilon(\theta)(\Upsilon(\A))$. We can thus write $\Upsilon(\theta)=\Theta_1\oplus \Theta_2$ where 
\[
\Theta_1:\Upsilon(\A)\to B(\H_1), \quad \Theta_2:\Upsilon(\A)\to B(\H_2)
\]
 are unital completely bounded homomorphisms. It is readily verified that $\Theta_1=\Upsilon(\theta_1)$ and it is clear that $\|\Theta_1\|_{cb}\leq r^2$. Another application of Theorem \ref{T:paulsenchar} shows that $\theta_1$ has Paulsen's similarity property with constant $r$.
\end{proof}

In light of Theorem \ref{T:paulsenchar}, it is easy to construct  homomorphisms $\theta$ that do not have Paulsen's similarity property with the smallest possible constant, namely $\|\theta\|_{cb}^{1/2}$.

\begin{example}\label{E:IdempotentGrows}
Let
\[
E= \left[\begin{array}{cc} 2 & -2 \\ 1 & -1 \end{array}\right]\in M_2(\bC)
\]
which is an idempotent. Consider the homomorphism $\theta : \bC \rightarrow M_2(\bC)$ uniquely determined by $\theta(1)=E$.
Then, $\theta$ is a completely bounded homomorphism with 
\[
\|\theta\|_{cb} = \|E\| = \left\|\left[\begin{array}{cc} 5 & -5 \\ -5 & 5 \end{array}\right]\right\|^{1/2} = \sqrt{10}. 
\]
Since $\bC$ is unital, we have that $\Upsilon(\bC)=\bC\oplus \bC$ and $\upsilon: \bC \rightarrow \Upsilon(\bC)$ is given by 
\[
\upsilon(\lambda) = \lambda \oplus 0, \quad \lambda\in \bC.
\]
We find
\begin{align*}
\Upsilon(\theta)(1\oplus -1) &= \Upsilon(\theta)(2\oplus 0) - \Upsilon(\theta)(1\oplus 1) =\theta(2)-I
\\ & = \left[\begin{array}{cc} 4 & -4 \\ 2 & -2 \end{array}\right] - \left[\begin{array}{cc} 1 & 0 \\ 0 & 1\end{array}\right]
 = \left[\begin{array}{cc} 3& -4 \\ 2 & -3\end{array}\right]
\end{align*}
and so
\begin{align*}
\|\Upsilon(\theta)\|_{cb} & \geq  \|\Upsilon(\theta)(1 \oplus -1)\| =  \left\|\left[\begin{array}{cc} 13 & -18 \\ -18 & 25\end{array}\right]\right\|^{1/2}
\\ & >  \|\theta\|_{cb}.
\end{align*}
By Theorem \ref{T:paulsenchar}, we conclude that $\theta$ does not have Paulsen's similarity property with constant $r = \|\theta\|_{cb}^{1/2}$.

\end{example}

\section{A subclass of the completely bounded maps}\label{S:Fr}

The main issue behind the shortcomings of the class $\CB_r(\M)$ exhibited in Section \ref{S:CBext} is the lack of a perfect analogue of Stinespring's dilation theorem. Even though Theorems \ref{T:paulsencbdilation1} and \ref{T:paulsencbdilation2} are useful replacements, the fact remains that a completely bounded linear map does not necessarily dilate to a completely bounded homomorphism \emph{of the same norm}, which conflicts with the machinery developed in Section \ref{S:extmachine}. In this section, we attempt to remedy this problem by restricting our attention to a smaller class of maps. 

We recall at the onset that \emph{we always assume that operator spaces are concretely represented on some Hilbert space}. This will be a standing assumption throughout this section and the next. In particular, it makes sense to consider the $C^*$-algebra generated by an operator space, although this depends on the choice of representation. This dependence is not relevant for the purposes of Sections \ref{S:Fr} and \ref{S:Frext}. We will carefully analyze the impact of the choice of representation in Section \ref{S:C*env}. 

Let $\M$ be an operator space and let $r\geq 1$. Given a Hilbert space $\H$,  we denote by $\Pau_r(\M,\H)$ the set of linear maps $\phi:\M\to B(\H)$ for which there is a Hilbert space $\K_\phi$ containing $\H$ along with a completely bounded homomorphism $\theta:C^*(\M)\to B(\K_\phi)$ that has Paulsen's similarity property with constant $r^{1/2}$ and such that
\[
\phi(a)=P_{\H}\theta(a)|_{\H}, \quad a\in \M.
\]
In view of Theorem \ref{T:paulsenchar}, the reader may venture to guess that $\phi\in \Pau_r(\M,\H)$ if and only if $\Upsilon(\phi)\in \CB_r(\Upsilon(\M),\H)$. However, simple examples show that such a description unfortunately does not hold (see Example \ref{E:UCBneqFr} below).

The remainder of this section is devoted to exhibiting some of the basic properties of the subclass $\Pau_r(\M)$.
We first note that it is clear that $\Pau_r(\M,\H)\subset\CB_r(\M,\H)$, and it follows from Theorems \ref{T:paulsencbdilation2} and \ref{T:paulsensim} that there is a positive constant $C_r>r$ depending only on $r$ such that
\[
\CB_r(\M,\H)\subset \Pau_{C_r}(\M,\H).
\]
This seems to indicate that the inclusion $\Pau_r(\M,\H)\subset\CB_r(\M,\H)$ may be strict. This is indeed the case.  Before proceeding with the example illustrating this fact, we note in passing that similar questions were considered in \cite{paulsensuen1985}.

\begin{example}\label{E:double}
Let $\H=\bC$, let $\M=B(\H)$ (so that $\M=C^*(\M)$) and let 
\[
\phi:\M\to B(\H)
\]
be the completely bounded linear map of multiplication by $r>1$, so that $\phi(\lambda)=r\lambda$ for $\lambda\in \M$. Let $\K$ be a Hilbert space containing $\H$ and let $\theta:C^*(\M)\to B(\K)$ be a homomorphism such that 
\[
\phi(\lambda)=P_{\H} \theta(\lambda)|_{\H}, \quad \lambda\in\M.
\]
Then, with respect to the decomposition $\K=\H\oplus (\K\ominus \H)$ we have
\[
\theta(1)=\left[
\begin{array}{cc}
r& a\\
b & c
\end{array}
\right].
\]
Since $\theta$ is multiplicative, we must have that $\theta(1)$ is idempotent, which is easily seen to force $b\neq 0$ as $r^2\neq r$. Hence
\[
\|\theta\|_{cb}\geq \|\theta(1)\|\geq \sqrt{\|b\|^2+r^2}>r=\|\phi\|_{cb}
\]
and we infer that $\phi\in \CB_r(\M,\bC)\setminus \Pau_r(\M,\bC)$.

\end{example}

Although we needed to choose $r>1$ in the example above, it can be shown that the inclusion $\Pau_1(\M,\H)\subset \CB_1(\M,\H)$ is still strict (see Example \ref{E:nilpotent} below). We note also that the basic idea behind the previous example is to exploit a multiplicative ``relation" within $\M$ that is not preserved by the linear map $\phi$.  This idea can be extended to identify an obstruction for a map in $\CB_r(\M)$ to lie inside the class $\Pau_r(\M)$.

\begin{theorem}\label{T:Frobs}
Let $\M$ be an operator space and let $\phi:\M\to B(\H_\phi)$ be a linear map with $\|\phi\|_{cb}=r$. Assume that there are two elements $a_0,b_0\in \M$ such that $\|a_0\|=\|b_0\|=1$, $b_0 a_0=0$, $\phi(a_0)=r I$ and $\phi(b_0)\neq 0$. Then, $\phi$ does not lie in $\Pau_r(\M,\H_\phi)$.
\end{theorem}
\begin{proof}
Let $\K_\phi$ be a Hilbert space containing $\H_\phi$ and let $\theta:C^*(\M)\to B(\K_\phi)$ be a homomorphism such that
\[
\phi(a)=P_{\H_\phi}\theta(a)|_{\H_\phi},\quad a\in \M.
\]
Since $b_0a _0=0$, we must have $\theta(b_0)\theta(a_0)=0$ and in particular 
\begin{align*}
0&=P_{\H_\phi}\theta(b_0)\theta(a_0)P_{\H_\phi}=\phi(b_0)\phi(a_0)+P_{\H_\phi}\theta(b_0)P_{\H^\perp_\phi}\theta(a_0)P_{\H_\phi}\\
&=r\phi(b_0)+P_{\H_\phi}\theta(b_0)P_{\H^\perp_\phi}\theta(a_0)P_{\H_\phi}.
\end{align*}
Hence, we find
\[
\|P_{\H_\phi}\theta(b_0)P_{\H^\perp_\phi}\|\|P_{\H^\perp_\phi}\theta(a_0)P_{\H_\phi}\|\geq r\|\phi(b_0)\|>0.
\]
Since $\phi(a_0)=rI$, we find
\begin{align*}
\|\theta(a_0)\|^2&\geq \|\theta(a_0)P_{\H_\phi}\|^2=\|\phi(a_0)+P_{\H^\perp_\phi}\theta(a_0)P_{\H_\phi}\|^2\\
&=r^2+\|P_{\H^\perp_\phi}\theta(a_0)P_{\H_\phi}\|^2>r^2
\end{align*}
and thus $\|\theta\|_{cb}>r$. We conclude that $\phi$ does not lie in $\Pau_r(\M,\H_\phi)$.
\end{proof}

Examples satisfying the conditions of the previous theorem are easily constructed. 

\begin{example}\label{E:nilpotent}
Let $\M$ be an operator space containing an operator $N$ such that $\|N\|=1$ and $N^2=0$. Choose a linear functional $\phi: \M\to \bC$ with $\|\phi\|=1$ and $\phi(N)=1$. Then, $\|\phi\|_{cb}=1$. Applying Theorem \ref{T:Frobs} with $a_0=b_0=N$ shows that $\phi\in \CB_1(\M,\bC)\setminus\Pau_1(\M,\bC)$. 
\end{example}

Obviously, the element $a_0$ from Theorem \ref{T:Frobs} can never be the identity of $\M$. Nevertheless, even if we insist that a map $\phi\in \CB_r(\M,\H)$ be unital, it still may not lie in $\Pau_r(\M,\H)$.

\begin{example}\label{E:UCBr}
Let $\M\subset M_2(\bC)$ be the unital operator space consisting of upper triangular Toeplitz matrices, that is
\[
\M=\left\{ 
\begin{bmatrix}
x & y \\
0 & x
\end{bmatrix}:x,y\in \bC
\right\}.
\]
For convenience, given $x,y\in \bC$ we use the following notation
\[
T_{x,y}=\begin{bmatrix}
x & y \\
0 & x
\end{bmatrix}.
\]
Let $\phi:\M\to \bC$ be the unital linear functional defined as
\[
\phi(T_{x,y})=x+y, \quad x,y\in \bC.
\]
If $x\neq 0$ and $y\neq 0$, a standard verification shows that
\[
\|T_{x,y}\|^2>|x|^2+|y|^2.
\]
Thus,
\[
|\phi(T_{x,y}) |=|x+y|\leq \sqrt{2}\sqrt{|x|^2+|y|^2}<\sqrt{2}\|T_{x,y}\|
\]
by the Cauchy-Schwarz inequality. Moreover, for every $x\in \bC$ and $y\in \bC$ we have
\[
|\phi(T_{0,y})|=|y|=\|T_{0,y}\|, \quad |\phi(T_{x,0})|=|x|=\|T_{x,0}\|.
\]
Hence, we conclude that $|\phi(T)|<\sqrt{2}\|T\|$ for every $T\in \M$ with $\|T\|=1$. By compactness, we infer that $\|\phi\|<\sqrt{2}$, and since $\phi$ is a functional we find $\phi\in \CB_{\sqrt{2}-\eps}(\M,\bC)$ for some $\eps>0$.

Let $\K$ be a Hilbert space containing $\H = \bC$ as a subspace. Assume that $\theta:C^*(\M)\to B(\K)$ is a homomorphism such that
\[
\phi(T)=P_{\H}\theta(T)|_{\H}, \quad T\in \M.
\]
We see that $\phi(T_{0,1})=1$  so we may write
\[
\theta(T_{0,1})=\begin{bmatrix}
1 & A\\
B& C
\end{bmatrix}.
\]
Since $\theta$ is multiplicative and $T_{0,1}^2=0$, we have $\theta(T_{0,1})^2=0$. In particular this forces $1+AB=0$. Thus, either $\|A\|\geq 1$ or $\|B\|\geq 1$, which implies that $\|\theta(T_{0,1})\|\geq \sqrt{2}$. Since $\|T_{0,1}\|=1$, we conclude that $\|\theta\|_{cb}\geq \sqrt{2}$. Thus, we see that 
\[
\phi\in \CB_{\sqrt{2}-\eps}(\M,\bC)\setminus \Pau_{\sqrt{2}-\eps}(\M,\bC).
\]

\end{example}

Before establishing a useful property of maps in the class $\Pau_r(\M)$, we need the following elementary calculation.

\begin{lemma}\label{L:dilationsim}
Let $\M$ be an operator space and let $\phi:\M\to B(\H_\phi)$ be a linear map. Assume that there is a Hilbert space $\K_\phi$ containing $\H_\phi$ and a linear map $\psi:\M\to B(\K_\phi)$ such that
\[
\phi(a)=P_{\H_\phi}\psi(a)|_{\H_\phi}, \quad a\in \M.
\]
Let $X\in B(\H_\phi)$ be invertible with $\|X\|=\|X^{-1}\|$. Then, there is another invertible operator $X'\in B(\K_\phi)$ for which the following assertions hold.
\begin{enumerate}

\item[\rm{(a)}] We have $\|X\|=\|X'\|=\|X^{-1}\|=\|X'^{-1}\|$.

\item[\rm{(b)}] The space $\H_\phi$ is reducing for $X'$ and $X'\H_\phi=\H_\phi=X'^*\H_\phi$.

\item[\rm{(c)}] We have
\[
X\phi(a)X^{-1}=P_{\H_\phi}X'\psi(a)X'^{-1}|_{\H_\phi}, \quad a\in \M.
\]
\end{enumerate}
\end{lemma}
\begin{proof}
Define the invertible operator $X': \K_\phi \to \K_\phi$ as $X'=X\oplus I$ according to the decomposition $\K_\phi=\H_\phi\oplus \H_\phi^\perp$. Properties (a) and (b) are clearly satisfied, so it remains to establish (c). We have 
\[
X=X'|_{\H_\phi} , \quad X^{-1}=X'^{-1}|_{\H_\phi}
\]
and
\[
X'P_{\H_\phi}=P_{\H_\phi}X', \quad P_{\H_\phi}X'^{-1}=X'^{-1}P_{\H_\phi}
\]
so we find
\begin{align*}
X\phi(a)X^{-1}&=X'\phi(a)X'^{-1}|_{\H_\phi}=X'P_{\H_\phi}\psi(a)P_{\H_\phi}X'^{-1}|_{\H_\phi}\\
&=P_{\H_\phi}X'\psi(a)X'^{-1}|_{\H_\phi}
\end{align*}
for every $a\in \M$.
\end{proof}

We now exhibit an important dilation property of maps in the class $\Pau_r$.

\begin{theorem}\label{T:Frdilation}
Let $\M$ be an operator space, let $r\geq 1$ and let $\phi\in \Pau_r(\M,\H_\phi)$. Then, there is a Hilbert space $\K_\phi$ containing $\H_\phi$, a $*$-homomorphism $\sigma: C^*(\M)\to B(\K_\phi)$ and an invertible operator $X\in B(\K_\phi)$ with 
\[
\|X\|=\|X^{-1}\|\leq r^{1/2}
\]
such that 
\[
\phi(a)=P_{\H_\phi} X\sigma(a)X^{-1}|_{\H_\phi},\quad a\in \M.
\]
\end{theorem}
\begin{proof}
By assumption, there is a Hilbert space $\K$ containing $\H_\phi$ and a homomorphism $\theta:C^*(\M)\to B(\K)$ that has Paulsen's similarity property with constant $r^{1/2}$ and such that
\[
\phi(a)=P_{\H_\phi} \theta(a)|_{\H_\phi},\quad a\in \M.
\]
Applying Corollary \ref{C:hom} to $\theta$, we obtain a Hilbert space $\K_\phi$ containing $\K$, a $*$-homomorphism $\sigma: C^*(\M)\to B(\K_\phi)$ and an invertible operator $X\in B(\K)$
with 
\[
\|X\|=\|X^{-1}\|\leq r^{1/2}
\]
such that
\[
X\theta(a)X^{-1}=P_{\K} \sigma(a)|_{\K},\quad a\in C^*(\M).
\]
By Lemma \ref{L:dilationsim}, we can find an invertible operator $X'\in B(\K_\phi)$ such that
\[
\|X'\|=\|X\|\leq r^{1/2}, \quad \|X'^{-1}\|=\|X^{-1}\|\leq r^{1/2}
\]
and
\[
\theta(a)=P_{\K}X'^{-1}\sigma(a)X'|_{\K}, \quad a\in C^*(\M).
\]
We conclude that
\[
\phi(a)=P_{\H_\phi} X'^{-1}\sigma(a)X'|_{\H_\phi},\quad a\in \M.
\]
\end{proof}

Given an operator space $\M$, a positive number $r$  and a Hilbert space $\H$, we denote by $\CP_r(\M,\H)$ the set of linear maps $\phi:\M\to B(\H)$ that admit a completely positive extension $\psi:C^*(\M)\to B(\H)$ satisfying $\|\psi\|_{cb}\leq r$. An interesting consequence of the previous theorem is an upper bound for the distance of an element of $\Pau_r(\M,\H)$ to the set $\CP_1(\M,\H)$.

Before stating it, we need a simple preliminary calculation. Let $\fA$ be a unital $C^*$-algebra and let $X\in \fA$ be a positive invertible operator with $\|X\|=\|X^{-1}\|$. Then, for each $a\in \fA$ we have
\begin{align*}
\|a-XaX^{-1}\|&\leq \|(I-X)a\|+\|Xa(I-X^{-1})\|\\\
&\leq \|a\|(\|I-X\|+\|X\| \|I-X^{-1}\|)\\
&\leq \|a\| (1+\|X\|) \max\{\|I-X\|,\|I-X^{-1}\|\}.
\end{align*}
Since $X$ is positive and satisfies $\|X\|=\|X^{-1}\|$, we may use the spectral theorem to conclude that
\[
\max\{\|I-X\|,\|I-X^{-1}\|\}\leq \max\{\|X\|-1,1-1/\|X\|\}
\]
so that
\[
\|a-XaX^{-1}\|\leq \|a\|(1+\|X\|)\max\{\|X\|-1,1-1/\|X\|\}.
\]
We can now establish the announced distance estimate.

\begin{corollary}\label{C:distance}
Let $\M$ be an operator space, let $r\geq 1$ and let $\phi\in \Pau_r(\M,\H)$. Then, there is a linear map $\psi\in \CP_1(\M,\H)$ with the property that
\[
\|\phi-\psi\|_{cb}\leq r-1.
\]
\end{corollary}
\begin{proof}
By Theorem \ref{T:Frdilation}, there is a Hilbert space $\K_\phi$ containing $\H$, a $*$-homomorphism $\sigma: C^*(\M)\to B(\K_\phi)$ and an invertible operator $X\in B(\K_\phi)$ with 
\[
\|X\|=\|X^{-1}\|\leq r^{1/2}
\]
such that 
\[
\phi(a)=P_{\H} X\sigma(a)X^{-1}|_{\H},\quad a\in \M.
\]
Using the polar decomposition if necessary, we may assume that $X$ is positive. Let now $\psi:\M\to B(\H)$ be the linear map defined as
\[
\psi(a)=P_{\H} \sigma(a)|_{\H},\quad a\in \M.
\]
Then, $\psi\in \CP_1(\M,\H)$. Using the calculation preceding the corollary, we see that 
\begin{align*}
\|\phi-\psi\|_{cb}&\leq (1+r^{1/2})\max\{r^{1/2}-1,1-1/r^{1/2}\}\\
&= (1+r^{1/2})(r^{1/2}-1)=r-1.
\end{align*}
\end{proof}

We can now show that $\Pau_1(\M,\H)=\CP_1(\M,\H)$ for every Hilbert space $\H$.

\begin{corollary}\label{C:F1}
Let $\M$ be an operator space and let $\phi:\M\to B(\H)$ be a linear map. Then, $\phi\in \Pau_1(\M,\H)$ if and only if $\phi$ admits a contractive completely positive extension to $C^*(\M)$.
\end{corollary}
\begin{proof}
If $\phi \in \Pau_1(\M,\H)$ then $\phi\in \CP_1(\M,\H)$ by Corollary \ref{C:distance}.
Conversely, assume that $\phi$ admits a contractive completely positive extension to $C^*(\M)$. By Lemma \ref{L:ccp}, there is a Hilbert space $\K_\phi$ containing $\H$ and $*$-homomorphism $\pi:C^*(\M)\to B(\K_\phi)$ such that
\[
\phi(a)=P_{\H}\pi(a)|_{\H}, \quad a\in \M.
\]
Thus, $\phi\in \Pau_1(\M,\H)$.

\end{proof}

We make a few remarks regarding the previous result. First, note that if $\M$ is unital and $\phi:\M\to B(\H)$ is a \emph{unital} completely contractive map, then $\phi$ admits a unital (hence contractive) completely positive extension to $C^*(\M)$, so that $\phi\in \Pau_1(\M,\H)$ by Corollary \ref{C:F1}. This shows that Example \ref{E:UCBr} is somewhat sharp: the norm of the functional $\phi$ cannot be taken to be $1$ therein. Moreover, we mention that Example \ref{E:double} shows that Corollary \ref{C:F1} fails beyond the completely contractive setting. Indeed, that example exhibits a completely positive map $\phi:\bC\to \bC$ with $\|\phi\|_{cb}=r>1$ such that $\phi$ does not lie in $\Pau_r(\bC,\bC)$. The reverse inclusion also fails as the next example illustrates.

\begin{example}\label{E:sim}
Let $\H$ be a Hilbert space and let $X\in B(\H)$ be a non-unitary invertible operator with 
\[
\|X\|=\|X^{-1}\|=r^{1/2}.
\]
In particular, $r>1$. Let $\theta:B(\H)\to B(\H)$ be defined as
\[
\theta(a)=XaX^{-1}, \quad a\in B(\H).
\]
It is clear that $\theta\in \Pau_r(B(\H),\H)$. Moreover, by applying $\theta$ to rank-one operators it is easily verified that $\|\theta\|=r>1$, so that $\|\theta\|>\|\theta(I)\|$. Thus, $\theta$ is not completely positive.
\end{example}

A useful consequence of Corollary \ref{C:F1} is the following corollary. It guarantees that certain maps preserve the classes $\Pau_r(\M)$.

\begin{corollary}\label{C:Frcomp}
Let $\M\subset B(\H_\M)$ and $\N\subset B(\H_\N)$ be operator spaces, let $\tau\in \Pau_1(\M,\H_\N)$ such that $\tau(\M)\subset \N$ and let $\phi\in \Pau_r(\N,\H_\phi)$. Then, $\phi\circ \tau\in \Pau_r(\M,\H_\phi)$. 
\end{corollary}
\begin{proof}

By assumption, there is a Hilbert space $\K_\phi$ containing $\H_\phi$ and a homomorphism $\theta:C^*(\N)\to B(\K_\phi)$ that has Paulsen's similarity property with constant $r^{1/2}$ and is such that
\[
\phi(a)=P_{\H_\phi}\theta(a)|_{\H_\phi}, \quad a\in \N.
\]
By Corollary \ref{C:hom}, there is an invertible operator $X\in B(\K_\phi)$ with $\|X\|=\|X^{-1}\|\leq r^{1/2}$ such that the homomorphism $\theta_X:C^*(\N)\to B(\K_\phi)$ defined as
\[
\theta_X(a)=X\theta(a)X^{-1}, \quad a\in C^*(\N)
\]
is contractive and completely positive. Next, use Corollary \ref{C:F1} to find a contractive completely positive map $\Psi:C^*(\M)\to B(\H_\N)$ that agrees with $\tau$ on $\M$. By Arveson's extension theorem there is a contractive completely positive map $\Xi:B(\H_\N)\to B(\K_\phi)$ that agrees with $\theta_X$ on $C^*(\N)$. Then, $\Xi\circ \Psi$ is contractive and completely positive, and it agrees with $\theta_X\circ \tau$ on $\M$. Hence, by Lemma \ref{L:ccp} we obtain a Hilbert space $\K$ containing $\K_\phi$ and a $*$-homomorphism $\pi:C^*(\M)\to B(\K)$ such that
\[
\theta_X\circ \tau(a)=P_{\K_\phi}\pi(a)|_{\K_\phi}, \quad a\in \M.
\]
By Lemma \ref{L:dilationsim}, we see that
\[
\theta\circ \tau(a)=P_{\K_\phi}\pi_X(a)|_{\K_\phi}, \quad a\in \M
\]
where $\pi_X:C^*(\M)\to B(\K)$ is a homomorphism that has Paulsen's similarity property with constant $r^{1/2}$. Finally, we note that
\[
\phi\circ \tau(a)=P_{\H_\phi}\theta\circ \tau(a)|_{\H_\phi}=P_{\H_\phi}\pi_X(a)|_{\H_\phi}, \quad a\in \M
\]
so that $\phi\circ \tau\in \Pau_r(\M,\H_\phi)$.
\end{proof}

We close this section by describing a relationship between the classes $\Pau_r(\M)$ for different values of the parameter $r\geq 1$. The key technical tool is the following, that we require in later sections as well. See Subsection \ref{SS:ultra} for some background on ultraproducts.

\begin{lemma}\label{L:Frultra}
Let $\M$ be an operator space, let $\Lambda$ be a directed set and let $\U$ be a cofinal ultrafilter on $\Lambda$. For each $\lambda\in \Lambda$, let $\phi_\lambda\in \Pau_{r_\lambda}(\M,\H_\lambda)$.  Assume that $(r_\lambda)_{\lambda\in \Lambda}$ is a bounded. Then, the ultraproduct $\lim_\U (\phi_\lambda)_{\lambda\in \Lambda}$ yields an element of $\Pau_r(\M,\H_\U)$, where $r=\lim_\U (r_\lambda)_{\lambda\in \Lambda}$ and $\H_\U$ is the ultraproduct Hilbert space of $(\H_\lambda)_{\lambda\in \Lambda}$ along $\U$.
\end{lemma}
\begin{proof}
For each $\lambda\in \Lambda$, there is a Hilbert space $\K_\lambda$ containing $\H_\lambda$ along with a completely bounded homomorphism $\theta_\lambda: C^*(\M)\to B(\K_\lambda)$ that has Paulsen's similarity property with constant $r_\lambda^{1/2}$ and such that 
\[
\phi_\lambda(a)=P_{\H_\lambda}\theta_\lambda(a)|_{\H_\lambda}, \quad a\in \M.
\]
Thus, for each $\lambda\in \Lambda$ there is an invertible operator $X_\lambda\in B(\K_\lambda)$ with
\[
\|X_\lambda\|=\|X_\lambda^{-1}\|\leq r_\lambda^{1/2}
\]
and such that the homomorphism $\Xi_\lambda:C^*(\M)\to B(\K_\lambda)$ defined as
\[
\Xi_\lambda(a)= X_\lambda\theta_\lambda(a)X_\lambda^{-1},\quad a\in C^*(\M)
\]
is completely contractive.

Let $X\in B(\K_\U)$ be defined as $\lim_\U(X_\lambda)_{\lambda\in \Lambda}$. Then, $X$ is invertible with $\|X\|\leq r^{1/2}$ and $\|X^{-1}\|\leq r^{1/2}$. Upon renormalizing, we may assume that $\|X\|=\|X^{-1}\|\leq r^{1/2}$. Let $\theta:C^*(\M)\to B(\K_\U)$ be the homomorphism defined as $\lim_\U (\theta_\lambda)_{\lambda\in \Lambda}$. Note  that
\[
X\theta(a)X^{-1}=\lim_\U (X_\lambda \theta_\lambda(a) X_\lambda^{-1})_{\lambda\in \Lambda}, \quad a\in C^*(\M)
\]
so that $\theta$ has Paulsen's similarity property with constant $r^{1/2}$. Finally, we observe that $\H_\U\subset \K_\U$, and if $\xi=(\xi_\lambda)_{\lambda\in \Lambda}, \eta=(\eta_\lambda)_{\lambda\in \Lambda}$ are two elements of $\prod_{\lambda\in \Lambda} \H_\lambda$ then we have
\begin{align*}
\langle  \theta(a)[\xi],[\eta]\rangle_{\K_\U}&=\lim_\U (\langle \theta_\lambda(a)\xi_\lambda,\eta_\lambda\rangle_{\K_\lambda})_{\lambda\in \Lambda}\\
&=\lim_\U (\langle \phi_\lambda(a)\xi_\lambda,\eta_\lambda\rangle_{\H_\lambda})_{\lambda\in \Lambda}\\
&=\langle (\lim_\U (\phi_\lambda)_{\lambda\in \Lambda})(a)[\xi],[\eta]\rangle_{\H_\U}.
\end{align*}
We conclude that $\lim_\U (\phi_\lambda)_{\lambda\in \Lambda}$ lies in $\Pau_r(\M,\H_\U)$.
\end{proof}

We can now describe a certain continuity property of the class $\Pau_r(\M)$ with respect to $r$.

\begin{theorem}\label{T:Frinter}
Let $\M$ be an operator space, let $r\geq 1$ and let $\H$ be a Hilbert space. Then,
\[
\Pau_r(\M,\H)=\cap_{\eps>0}\Pau_{r+\eps}(\M,\H).
\]
\end{theorem}
\begin{proof}

It follows from the definition that $\Pau_r(\M,\H)\subset \cap_{\eps>0}\Pau_{r+\eps}(\M,\H)$. Conversely, let $\phi\in\cap_{\eps>0}\Pau_{r+\eps}(\M,\H)$. 
Thus, $\phi\in\cap_n \Pau_{r+1/n}(\M,\H)$.
Let $\U$ be a cofinal ultrafilter on $\bN$. By Lemma \ref{L:Frultra} we see that $\lim_\U (\phi)_{n\in \bN}\in \Pau_r(\M,\H_\U)$. Finally, let $V:\H\to \H_\U$ be the isometry defined as
\[
V\xi=\lim_\U (\xi)_{n\in \bN}, \quad \xi \in \H.
\]
A standard verification yields
\[
\phi(a)=V^*\left(\lim_\U (\phi)_{n\in \bN}\right)(a) V, \quad a\in \M
\]
so an application of Lemma \ref{L:V} shows that indeed $\phi\in \Pau_r(\M,\H)$.
\end{proof}

\section{Extremals in the class $\Pau_r(\M)$}\label{S:Frext}

We emphasize once more that we always assume that operator spaces are concretely represented on some Hilbert space.

Let $\M$ be an operator space and let $r\geq 1$.  In this section we restrict the dilation order defined on the class $\CB_r(\M)$ in Section \ref{S:CBext} to the subclass $\Pau_r(\M)$.  We recall the notation and terminology here for convenience. Let $\phi\in \Pau_r(\M,\H_\phi)$ and $\psi\in \Pau_r(\M,\H_\psi)$. We write  $\phi \prec \psi$ if $\H_\phi\subset \H_\psi$ and  
\[
\phi(a)=P_{\H_\phi}\psi(a)|_{\H_\phi}, \quad a\in \M.
\]
Clearly, this is a dilation order on $\Pau_r(\M)$.
We say that an element $\omega \in \Pau_r(\M,\H_\omega)$ is $\Pau_r(\M)$-\emph{extremal} if whenever $\delta\in 
\Pau_r(\M,\H_\delta)$ satisfies $\omega\prec \delta$, we necessarily have that $\H_\omega$  is reducing for $\delta(\M)$. It is an easy consequence of Lemma \ref{L:dilationsim} that $\Pau_r(\M)$-extremal elements are preserved by unitary equivalence. This will be used throughout, often without mention. We will also require the following simple observation.

\begin{lemma}\label{L:Frextdirectsum}
Let $\M$ be an operator space and let $r\geq 1$. Let $\omega\in \Pau_r(\M,\H_\omega)$ be a $\Pau_r(\M)$-extremal element and let $\X\subset \H_\omega$ be a reducing subspace for $\omega(\M)$. If we define $\omega':\M\to B(\X)$ as
\[
\omega'(a)=\omega(a)|_\X, \quad a\in \M
\]
then $\omega'\in\Pau_r(\M,\X)$ and it is a $\Pau_r(\M)$-extremal element.
\end{lemma}
\begin{proof}
It is straightforward to see that $\omega'\in\Pau_r(\M,\X)$. Let $\delta'\in \Pau_r(\M,\H_\delta)$ such that $\omega'\prec \delta'$ and define $\delta:\M\to \H_\delta\oplus (\H_\omega\ominus \X)$ as
\[
\delta(a)=\delta'(a)\oplus \omega(a)|_{(\H_\omega\ominus \X)}, \quad a\in \M.
\]
Upon invoking Corollary \ref{C:paulsendirectsum}, it is readily verified that $\delta\in \Pau_r(\M, \H_\delta\oplus (\H_\omega\ominus \X))$. Moreover, we note that $\omega\prec \delta$. Hence, $\H_\omega$ is reducing for $\delta(\M)$, which implies in particular that $\X$ is reducing for $\delta'(\M)$.
\end{proof}

Our immediate goal is to establish the existence of $\Pau_r(\M)$-extremal elements. For that purpose, we first show that $\Pau_r(\M)$ has the limit property with respect to the dilation order $\prec$. This is the first instance where our working with this smaller class of linear maps creates difficulties that are not present in the standard setting of \cite{dritschel2005}, \cite{davidson2015} and \cite{FHL2016}. Indeed, on top of the usual inductive limit procedure of Lemma \ref{L:chaincc}, we have to use the ultraproduct machinery from Subsection \ref{SS:ultra} to keep track of the multiplicative dilations.

\begin{lemma}\label{L:Frlimit}
Let $\M$ be an operator space and let $r\geq 1$. Then, the class $\Pau_r(\M)$ has the limit property with respect to the dilation order $\prec$.
\end{lemma}
\begin{proof}

Let $\Lambda$ be a totally ordered set. For each $\lambda\in \Lambda$, let $\phi_\lambda\in \Pau_r(\M,\H_\lambda).$ Assume that $\H_\lambda\subset \H_\mu$ and
\[
\phi_\lambda(a)=P_{\H_\lambda} \phi_\mu(a)|_{\H_\lambda}, \quad a\in \M
\]
whenever $\mu\geq \lambda$. Set $\H=\ol{\cup_\lambda \H_\lambda}$. We need to find an element $\psi\in \Pau_r(\M,\H)$ such that $\phi_\lambda\prec \psi$ for every $\lambda\in \Lambda$. Let $\psi:\M\to B(\H)$ be the map constructed from the collection $(\phi_\lambda)_{\lambda\in\Lambda}$ as in the proof of Lemma \ref{L:chaincc}. It is clear that $\phi_\lambda\prec \psi$ for every $\lambda\in \Lambda$. We need only verify that $\psi\in \Pau_r(\M,\H)$.

Let $\U$ be a cofinal ultrafilter on $\Lambda$ and let $\H_\U$ be the ultraproduct Hilbert space of $(\H_\lambda)_{\lambda\in \Lambda}$ along $\U$.   For each $\mu\in \Lambda$ we define an isometry $V_\mu:\H_\mu\to \H_\U$ as follows. If $\xi\in \H_\mu$ then $V_\mu \xi=[(\xi_\lambda)_{\lambda\in \Lambda}]$, where 
\[
\xi_{\lambda}=
\begin{cases}
\xi & \text{ if }  \lambda\geq \mu\\
0 & \text{ if }  \lambda<\mu.
\end{cases}
\]
Using that $\U$ is cofinal, it is easily verified that there exists another isometry $V:\H\to \H_\U$ such that $V\xi=V_\mu \xi$ if $\xi\in \H_\mu$. Now, let $\phi:\M\to B(\H_\U)$ be defined as $\phi=\lim_\U (\phi_{\lambda})_{\lambda\in \Lambda}$. We show that
\[
\psi(a)=V^*\phi(a)V, \quad a\in \M.
\]
Indeed, assume that $\xi,\eta\in \H_\mu$. Then, using again the fact that $\U$ is cofinal we find
\begin{align*}
\langle V^*\phi(a)V \xi,\eta\rangle_{\H}&=\langle \phi(a)V \xi,V\eta\rangle_{\H_\U}=\lim_\U \left( \langle \phi_\lambda(a)\xi,\eta\rangle_{\H_\lambda}\right)_{\lambda\in \Lambda}\\
&=\lim_\U \left( \langle \phi_\mu(a)\xi,\eta\rangle_{\H_\mu}\right)_{\lambda\in \Lambda}=\langle \phi_\mu(a)\xi,\eta\rangle_{\H_\mu}=\langle \psi(a)\xi,\eta\rangle_{\H}.
\end{align*}
Since $\H=\ol{\cup_\mu \H_\mu}$ we conclude that 
\[
\psi(a)=V^*\phi(a)V, \quad a\in \M
\]
as claimed. By Lemma \ref{L:Frultra} we see that $\phi\in \Pau_r(\M,\H_\U)$, so that an application of Lemma \ref{L:V} shows that $\psi\in \Pau_r(\M,\H)$.
\end{proof}

The following is now a straightforward consequence.

\begin{theorem}\label{T:Frext}
Let $\M$ be an operator space, let $r\geq 1$ and let $\phi\in \Pau_r(\M,\H_\phi)$. Then, there exists a $\Pau_r(\M)$-extremal element $\omega\in \Pau_r(\M,\H_\omega)$ such that $\phi\prec \omega$ and such that
\[
\dim \H_\omega\leq   (1+\aleph_0 \dim \M) \dim \H_\phi
\]
\end{theorem}
\begin{proof}
The class $\Pau_r(\M)$ has the limit property with respect to the dilation order $\prec$ by Lemma \ref{L:Frlimit}. By Theorem \ref{T:extconstruction}, we see that there is a maximal element $\zeta\in \Pau_r(\M,\K)$ such that $\phi\prec \zeta$. We claim that $\zeta$ is $\Pau_r(\M)$-extremal. Indeed, let $\delta\in \Pau_r(\M,\H_\delta)$ such that $\zeta\prec \delta$. We calculate for every $a\in \M$ and $\xi\in \K$ that
\begin{align*}
\|(\delta(a)-\zeta(a))\xi\|^2&=\| \delta(a)\xi\|^2+\|\zeta(a)\xi \|^2-2\re \langle \delta(a)\xi,\zeta(a)\xi\rangle\\
&=\|\delta(a)\xi \|^2-\| \zeta(a)\xi\|^2=0
\end{align*}
so that $\delta(a)\xi=\zeta(a)\xi$. Likewise,
\begin{align*}
\|(\delta(a)^*-\zeta(a)^*)\xi\|^2&=\| \delta(a)^*\xi\|^2+\|\zeta^*(a)\xi \|^2-2\re \langle \delta(a)^*\xi,\zeta(a)^*\xi\rangle\\
&=\|\delta(a)^*\xi \|^2-\|\zeta(a)^*\xi\|^2=0
\end{align*}
so that $\delta(a)^*\xi=\zeta(a)^*\xi$.
Hence, $\delta(\M)\K\subset \K$ and $\delta(\M)^*\K\subset \K$, so indeed $\zeta$ is $\Pau_r(\M)$-extremal. 

Let $\H_\omega=\ol{\H_\phi+C^*(\zeta(\M))\H_\phi}$. This is the smallest reducing subspace for $\zeta(\M)$ that contains $\H_\phi$. By choosing a Hamel basis for $\zeta(\M)$, we can find a dense subset of $C^*(\zeta(\M))$ with cardinality at most $\aleph_0 \dim \zeta(\M)$. Hence, there is a total subset of $\H_\omega$ with cardinality at most
\[
\dim \H_\phi+\aleph_0 \dim \zeta(\M) \dim \H_\phi.
\]
By applying the Gram-Schmidt algorithm, we find
\begin{align*}
\dim \H_\omega &\leq   (1+\aleph_0 \dim \zeta(\M)) \dim \H_\phi\leq  (1+\aleph_0 \dim\M) \dim \H_\phi.
\end{align*}
We define $\omega:\M\to B(\H)$ as
\[
\omega(a)=\zeta(a)|_{\H_\omega}, \quad a\in\M.
\]
It is clear that $\phi \prec \omega$, and it follows from Lemma \ref{L:Frextdirectsum} that $\omega\in \Pau_r(\M,\H_\omega)$ is $\Pau_r(\M)$-extremal.
\end{proof}

For our purposes, it will be relevant to know if the collection of $\Pau_r(\M)$-extremals ``completely norms" the operator space $\M$, in the following precise sense.

\begin{corollary}\label{C:Frextnorm}
Let $\M$ be an operator space and let $r\geq 1$. Then, there is a $\Pau_r(\M)$-extremal element $\omega\in\Pau_r(\M,\H_\omega)$  such that for every $n\in \bN$ and every $a\in M_n(\M)$ we have $\|a\|\leq \|\omega^{(n)}(a)\|$. Moreover, we have that
\[
\dim \H_\omega\leq(\dim \M)^2 (1+ \aleph_0\dim \M).
\]
\end{corollary}
\begin{proof}
There is an isometric $*$-homomorphism $\pi:C^*(\M)\to B(\H)$. Using a Hamel basis, we may find a dense subset $\S\subset C^*(\M)$ with 
\[
\card \S\leq \aleph_0 \dim \M.
\]
For each $a\in \S$ and $m\in \bN$, choose a unit vector $\xi_{a,m}\in \H$ such that
\[
\|\pi(a)\xi_{a,m}\|\geq \|a\|(1-1/m).
\]
Let $\H'$ denote the smallest closed subspace containing the sets
\[
\{\xi_{a,m}:a\in \S,m\in \bN\}
\]
and 
\[
 \{\pi( C^*(\M))\xi_{a,m}:a\in \S, m\in \bN\}.
\]
Then, $\H'$ is reducing for $\pi(C^*(\M))$. Let $\pi':C^*(\M)\to B(\H')$ be the $*$-homomorphism defined as
\[
\pi'(a)=\pi(a)|_{\H'}, \quad a\in C^*(\M).
\]
By construction, we see that $\pi'$ is isometric, and hence completely isometric. Now, $\H'$ contains a total subset of cardinality at most
\[
\aleph_0 \card \S+\aleph_0 (\card \S)^2.
\]
By applying the Gram-Schmidt algorithm, we find
\[
\dim \H'\leq \aleph_0 \card \S+\aleph_0 (\card \S)^2\leq \aleph_0 (\dim \M)^2.
\]
Let $\phi:\M\to B(\H')$ be defined as
\[
\phi(a)=\pi'(a), \quad a\in \M.
\] 
Clearly, we have that $\phi\in \Pau_r(\M,\H')$. By Theorem \ref{T:Frext}, there is a $\Pau_r(\M)$-extremal element $\omega\in \Pau_r(\M,\H_\omega)$ such that $\phi \prec \omega$ and 
\[
\dim \H_\omega \leq (1+ \aleph_0\dim \M) \dim \H'\leq  (\dim \M)^2 (1+ \aleph_0\dim \M) .
\]
Since $\prec$ is a dilation order, we have 
\[
\|\omega^{(n)}(a)\|\geq \|\phi^{(n)}(a)\|= \|a\|
\]
for every $n\in \bN$ and every $a\in M_n(\M)$.
\end{proof}

Let $\M$ be an operator space, let $\A_\M$ denote the operator algebra it generates, and let $r\geq 1$.
 An element $\phi\in \Pau_r(\M,\H_\phi)$ is said to have the \emph{unique extension property relative to} $\Pau_r(\M)$ if there exits a unique element $\Phi\in \Pau_r(\A_\M,\H_\phi)$ that agrees with $\phi$ on $\M$, and if this unique map $\Phi$ is a homomorphism that has Paulsen's similarity property with constant $r^{1/2}$. It is clear that  $\phi\in \Pau_r(\M,\H_\phi)$ has the unique extension property relative to $\Pau_r(\M)$ if and only if the following two statements hold:
\begin{enumerate}
\item[\rm{(a)}] there is $\Phi\in \Pau_r(\A_\M,\H_\phi)$ that agrees with $\phi$ on $\M$, and

\item[\rm{(b)}] every $\Psi\in \Pau_r(\A_\M,\H_\phi)$ that agrees with $\phi$ on $\M$ is a homomorphism that has Paulsen's similarity property with constant $r^{1/2}$.
\end{enumerate}

We now arrive at an important property of $\Pau_r(\M)$-extremals.

\begin{theorem}\label{T:Frextmult}
Let $\M$ be an operator space, let $r\geq 1$ and let $\omega\in \Pau_r(\M,\H_\omega)$ be a $\Pau_r(\M)$-extremal element. Then, $\omega$ has the unique extension property relative to $\Pau_r(\M)$.
\end{theorem}
\begin{proof}
By assumption, there is a Hilbert space $\K$ containing $\H_\omega$ along with a homomorphism $\theta:C^*(\M)\to B(\K)$ that has Paulsen's similarity property with constant $r^{1/2}$ and such that
\[
\omega(a)=P_{\H_\omega} \theta(a)|_{\H_\omega},\quad a\in \M.
\]
The map $\Omega:\A_\M\to B(\H_\omega)$ defined as
\[
\Omega(a)=P_{\H_\omega} \theta(a)|_{\H_\omega},\quad a\in \A_\M
\]
clearly lies in $\Pau_r(\A_\M,\H_\omega)$ and agrees with $\omega$ on $\M$. It remains to prove that any such extension is a homomorphism that has Paulsen's similarity property with constant $r^{1/2}$. 

For that purpose, let $\Psi\in \Pau_r(\A_\M,\H_\omega)$  that agrees with $\omega$ on $\M$.
Then, there is a Hilbert space $\K_\omega$ containing $\H_\omega$ along with a homomorphism $\rho:C^*(\M)\to B(\K_\omega)$ that has Paulsen's similarity property with constant $r^{1/2}$ and such that
\[
\Psi(a)=P_{\H_\omega} \rho(a)|_{\H_\omega},\quad a\in \A_\M.
\]
Let $\delta$ denote the restriction of $\rho$ to $\M$, so that $\delta\in \Pau_r(\M,\K_\omega)$ and $\omega\prec \delta.$ Since $\omega$ is assumed to be $\Pau_r(\M)$-extremal, we conclude that $\H_\omega$ is reducing for $\delta(\M)$, and hence for $\rho(\A_\M)$. Thus, we have
\[
\Psi(a)= \rho(a)|_{\H_\omega},\quad a\in \A_\M
\]
which shows that $\Psi$ is multiplicative. Invoking Corollary \ref{C:paulsendirectsum} we see that $\Psi$ has Paulsen's similarity property with constant $r^{1/2}$. 
\end{proof}

Although we will not need this fact, we remark that the argument used in the previous proof can be used to show that the unique multiplicative extension of $\omega$ to $\A_\M$ must be $\Pau_r(\A_\M)$-extremal. We leave the simple details to the reader.

We emphasize here that Theorem \ref{T:Frextmult} justifies our considering the subclass $\Pau_r(\M)$ rather than the full class $\CB_r(\M)$, since in this context \emph{every} extremal extends to be multiplicative on $\A_\M$.  This stands in contrast with the situation for general completely bounded linear maps, as illustrated by Theorem \ref{T:extmultcc}.

Next, we examine another unique extension property. An element $\phi\in\CP_1(\M,\H_\phi)$ is said to have the \emph{unique extension property relative to} $\CP_1(\M)$ if there exists a unique element $\Phi\in \CP_1(C^*(\M),\H_\phi)$ that agrees with $\phi$ on $\M$, and if moreover $\Phi$ is a $*$-homomorphism. By definition, an element $\phi\in \CP_1(\M,\H_\phi)$ has at least one contractive completely positive extension to $C^*(\M)$. Hence, $\phi\in \CP_1(\M,\H_\phi)$ has the unique extension property relative to $\CP_1(\M)$ if and only if every $\Psi\in \CP_1(C^*(\M),\H_\phi)$ that agrees with $\phi$ on $\M$ is a $*$-homomorphism.

We can now refine Theorem \ref{T:Frextmult}.

\begin{theorem}\label{T:Frextsimuep}
Let $\M$ be an operator space and let $r\geq 1$. Let $\omega\in \Pau_r(\M,\H_\omega)$ be $\Pau_r(\M)$-extremal. Then, the following statements hold.

\begin{enumerate}

\item[\rm{(1)}] There is an invertible operator $X\in B(\H_\omega)$ with $\|X\|=\|X^{-1}\|\leq r^{1/2}$  such that the map $\omega_X:\M\to B(\H_\omega)$ defined as
\[
\omega_X(a)=X\omega(a)X^{-1}, \quad a\in \M
\]
belongs to  $\Pau_1(\M,\H_\omega)$.

\item[\rm{(2)}] Let $Y\in B(\H_\omega)$ be an invertible operator with $\|Y\|=\|Y^{-1}\|\leq r^{1/2}$ and let $\omega_Y:\M\to B(\H_\omega)$ be defined as
\[
\omega_Y(a)=Y\omega(a)Y^{-1}, \quad a\in \M.
\]
If the map $\omega_Y$ belongs to $\Pau_1(\M,\H_\omega)$, then it has the unique extension property relative to $\CP_1(\M)$.

\end{enumerate}
\end{theorem}
\begin{proof}

(1) By Theorem \ref{T:Frextmult} there is a homomorphism $\Omega:\A_\M\to B(\H_\omega)$ that has Paulsen's similarity property with constant $r^{1/2}$ and that agrees with $\omega$ on $\M$. By Corollary \ref{C:hom}, there is an invertible operator $X\in B(\H_\omega)$ such that $\|X\|=\|X^{-1}\|\leq r^{1/2}$, a Hilbert space $\K_\omega$ containing $\H_\omega$ and a $*$-homomorphism $\pi:C^*(\M)\to B(\K_\omega)$ such that 
\[
\omega_X(a)=X\omega(a)X^{-1}=P_{\H_\omega}\pi(a)|_{\H_\omega}, \quad a\in \M.
\]
In particular, we have that $\omega_X\in \CP_1(\M,\H_\omega)$ and thus $\omega_X\in \Pau_1(\M,\H_\omega)$ by Corollary \ref{C:F1}.

(2) We see that $\omega_Y\in \CP_1(\M,\H_\omega)$ by Corollary \ref{C:F1}.  Let $\Psi:C^*(\M)\to B(\H_\omega)$ be a contractive completely positive map  that agrees with $\omega_Y$ on $\M$. We need to show that $\Psi$ is a $*$-homomorphism. By Lemma \ref{L:ccp}, there is a Hilbert space $\K$ containing $\H_\omega$ and a $*$-homomorphism $\pi:C^*(\M)\to B(\K)$ such that
\[
\Psi(a)=P_{\H_\omega}\pi(a)|_{\H_\omega}, \quad a\in C^*(\M).
\]
In particular, we see that
\[
\omega_Y(a)=P_{\H_\omega}\pi(a)|_{\H_\omega}, \quad a\in \M.
\]
We can now invoke Lemma \ref{L:dilationsim} to find another invertible operator $Y'\in B(\K)$ with $\|Y'\|=\|Y'^{-1}\|\leq r^{1/2}$, such that the space $Y'\H_\omega=\H_\omega=Y'^*\H_\omega$ is reducing for $Y'$ and  such that
\[
\omega(a)=Y^{-1}\omega_Y(a)Y=P_{\H_\omega}\pi_Y(a)|_{\H_\omega}, \quad a\in \M
\]
where $\pi_Y:C^*(\M)\to B(\K)$ is defined as
\[
\pi_Y(a)=Y'^{-1}\pi(a)Y', \quad a\in C^*(\M).
\]
If we denote by $\delta$ the restriction of $\pi_Y$ to $\M$, then $\delta\in \Pau_r(\M,\K)$ and $\omega\prec \delta$. Since $\omega$ is assumed to be $\Pau_r(\M)$-extremal, we conclude that 
\[
Y'^{-1}\pi(\M)Y' \H_\omega\subset \H_\omega, \quad (Y'^{-1}\pi(\M)Y')^* \H_\omega\subset \H_\omega
\]
or 
\[
\pi(\M)Y' \H_\omega\subset Y'  \H_\omega, \quad \pi(\M)^*Y'^{*-1} \H_\omega\subset Y'^{*-1} \H_\omega.
\]
Since $\pi$ is a $*$-homomorphism and $Y'\H_\omega=\H_\omega=Y'^{*-1}\H_\omega$, this translates to
\[
\pi(\M)\H_\omega\subset \H_\omega, \quad \pi(\M^*)\H_\omega\subset \H_\omega
\]
and thus $\pi(C^*(\M))\H_\omega\subset \H_\omega$. But recall that
\[
\Psi(a)=P_{\H_\omega}\pi(a)|_{\H_\omega}, \quad a\in C^*(\M)
\]
so that $\Psi$ is a $*$-homomorphism.
\end{proof}

Given an operator space $\M$ and two linear maps 
\[
\phi\in \CP_1(\M,\H_\phi),\quad \psi\in \CP_1(M,\H_\psi)
\]
we use our standard notation $\phi\prec\psi$ to mean that $\H_\phi\subset \H_\psi$ and
\[
\phi(a)=P_{\H_\phi}\psi(a)|_{\H_\phi},\quad a\in \M.
\]
A linear map $\omega\in \CP_1(\M,\H_\omega)$ is said to be $\CP_1(\M)$\emph{-extremal} if whenever $\delta\in \CP_1(\M,\H_\delta)$ satisfies $\omega \prec \delta$, we have that $\H_\omega$ is reducing for $\delta(\M)$. Interestingly, $\CP_1(\M)$-extremality and the unique extension property relative to $\CP_1(\M)$ are equivalent. This is reminiscent of the classical setting of unital completely positive maps \cite{dritschel2005}, and the proof is very similar. We provide it below for completeness.

\begin{lemma}\label{L:CP1uepext}
Let $\M$ be an operator space and let $\omega\in \CP_1(\M,\H_\omega)$. Then, $\omega$ is $\CP_1(\M)$-extremal if and only if it has the unique extension property relative to $\CP_1(\M)$.
\end{lemma}
\begin{proof}
Throughout the proof we assume that $\M\subset B(\H_\M)$.

Assume first that $\omega$  is $\CP_1(\M)$-extremal. Let $\Psi:C^*(\M)\to B(\H_\omega)$ be a contractive completely positive map that agrees with $\omega$ on $\M$. We must show that $\Psi$ is a $*$-homomorphism. By Lemma \ref{L:ccp}, there is a Hilbert space $\K_\omega$ containing $\H_\omega$ and a $*$-homomorphism $\pi:C^*(\M)\to B(\K_\omega)$ such that
\[
\Psi(a)=P_{\H_\omega}\pi(a)|_{\H_\omega}, \quad a\in C^*(\M).
\]
In particular, we have
\[
\omega(a)=P_{\H_\omega}\pi(a)|_{\H_\omega}, \quad a\in \M.
\]
By assumption, we see that $\H_\omega$ is reducing for $\pi(\M)$, and hence for $\pi(C^*(\M))$. This immediately implies that $\Psi$ is a $*$-homomorphism.

Conversely, assume that $\omega$ has the unique extension property relative to $\CP_1(\M)$. Let $\delta\in \CP_1(\M,\H_\delta)$ be such that $\omega\prec \delta$. There is a contractive completely positive map $\Delta:C^*(\M)\to B(\H_\delta)$ that agrees with $\delta$ on $\M$. By Lemma \ref{L:ccp}, there is a Hilbert space $\K_\delta$ containing $\H_\delta$ and a $*$-homomorphism $\pi:C^*(\M)\to B(\K_\delta)$ such that
\[
\Delta(a)=P_{\H_\delta}\pi(a)|_{\H_\delta}, \quad a\in C^*(\M).
\]
If we let $\Psi:C^*(\M)\to B(\H_\omega)$ be defined as
\[
\Psi(a)=P_{\H_\omega}\Delta(a)|_{\H_\omega}, \quad a\in C^*(\M)
\]
then we see that $\Psi$ is contractive and completely positive, and it agrees with $\omega$ on $\M$. 
By assumption, we see that $\Psi$ is a $*$-homomorphism. Let now $a\in C^*(\M)$. Applying the Schwarz inequality for completely positive maps \cite[Exercise 3.4]{paulsen2002} to a contractive completely positive extension of $\Delta$ to $B(\H_\M)$, we obtain 
\[
\Delta(a)^*\Delta(a)\leq \Delta(a^*a)
\]
whence
\[
P_{\H_\omega}\Delta(a)^*\Delta(a)P_{\H_\omega}\leq P_{\H_\omega}\Delta(a^*a)P_{\H_\omega}
\]
and thus
\begin{align*}
\Psi(a)^*\Psi(a)+P_{\H_\omega}\pi(a)^*P_{\H_\omega^\perp} \pi(a)P_{\H_\omega}&\leq \Psi(a^*a).
\end{align*}
But $\Psi$ is a $*$-homomorphism so that
\[
P_{\H_\omega}\pi(a)^*P_{\H_\omega^\perp} \pi(a)P_{\H_\omega}\leq 0
\]
or $P_{\H_\omega^\perp} \pi(a)|_{\H_\omega}=0$. Since $a\in C^*(\M)$ is arbitrary, we conclude that $\pi(C^*(\M))\H_\omega\subset \H_\omega$ and in particular $\delta(\M)\H_\omega\subset \H_\omega$ and $\delta(\M)^*\H_\omega\subset \H_\omega$, so that $\H_\omega$ is reducing for $\delta(\M)$. Hence, $\omega$ is $\CP_1(\M)$-extremal.
\end{proof}

We saw in Theorem \ref{T:Frextmult} that $\Pau_r(\M)$-extremals have the unique extension property relative to $\Pau_r(\M)$. We will see in Example \ref{E:extsim1} below that the converse is false. Hence, the previous lemma does not extend to the class $\Pau_r(\M)$ for $r>1$. 

The following is an easy consequence of Corollary \ref{C:F1}.

\begin{corollary}\label{C:Frextsimbdry}

Let $\M$ be an operator space and let $r\geq 1$. Let $\omega\in \Pau_r(\M,\H_\omega)$ be $\Pau_r(\M)$-extremal.  Let $Y\in B(\H_\omega)$ be an invertible operator  such that $\|Y\|=\|Y^{-1}\|\leq r^{1/2}$ and let $\omega_Y:\M\to B(\H_\omega)$ be defined as
\[
\omega_Y(a)=Y\omega(a)Y^{-1}, \quad a\in \M.
\]
If the map $\omega_Y$ belongs to $\Pau_1(\M,\H_\omega)$, then it is $\Pau_1(\M)$-extremal.
\end{corollary}
\begin{proof}
Simply combine Corollary \ref{C:F1}, Theorem \ref{T:Frextsimuep} and Lemma \ref{L:CP1uepext}.
\end{proof}

Theorem \ref{T:Frextsimuep} along with Lemma \ref{L:CP1uepext} shows that any $\Pau_r(\M)$-extremal is similar to a $\CP_1(\M)$-extremal, and it is natural to wonder whether the converse holds. Our next task is to show that this is not the case. First, we need the following result which is of independent interest. It shows that if $\omega$ is $\Pau_r(\M)$-extremal, then $\|\omega\|_{cb}$ cannot be too small, and $\omega$ does not belong to $\Pau_s(\M)$ for any $s<r$. This fact is not obvious from the definition of extremality.

\begin{theorem}\label{T:Frextbig}
Let $\M$ be an operator space, let $r\geq 1$ and let $\omega\in \Pau_r(\M,\H_\omega)$. Assume that $\omega$ is a non-zero $\Pau_r(\M)$-extremal element. Then, the following statements hold.
\begin{enumerate}
\item[\rm{(1)}] The map $\omega$ does not belong to $\Pau_s(\M,\H_\omega)$ if $1\leq s<r$. In particular, the unique homomorphism $\Omega\in \Pau_r(\A_\M,\H_\omega)$ that agrees with $\omega$ on $\M$ satisfies $\|\Upsilon(\Omega)\|_{cb}=r$.

\item[\rm{(2)}] The map $\omega$ satisfies
\[
\|\omega\|_{cb}\geq \frac{r}{16}-\frac{5}{4}.
\]

\end{enumerate}

\end{theorem} 
\begin{proof}
(1) Let $1\leq s<r$ and suppose on the contrary that $\omega\in \Pau_s(\M, \H_\omega)$. Then, there is a Hilbert space $\K_\omega$ containing $\H_\omega$ along with a homomorphism $\theta:\A_\M\to B(\K_\omega)$ that has Paulsen's similarity property with constant $s^{1/2}$ and such that
\[
\omega(a)=P_{\H_\omega}\theta(a)|_{\H_\omega}, \quad a\in \M.
\]
Let $\eps>0$ and define a linear map $\rho_\eps:\A_\M\to B(\K_\omega^{(4)})$ as
\[
\rho_\eps(a)=
\begin{bmatrix}
\theta(a) & 0 & \eps\theta(a) & 0\\
0 & \theta(a) & 0 & 0\\
0 & 0 & 0 & 0\\
0 & \eps\theta(a) & 0 & 0
\end{bmatrix}, \quad a\in \A_\M.
\]
It is readily verified that $\rho_\eps$ is a homomorphism with $\|\rho_\eps\|_{cb}\leq (1+2\eps)\|\theta\|_{cb}$. 
In addition, if we let 
\[
\Theta: \A_\M\to B(\K_\omega^{(4)})
\]
be defined as
\[
\Theta(a)=\theta(a)\oplus \theta(a)\oplus 0\oplus 0, \quad a\in \A_\M
\]
then
\[
\lim_{\eps\to 0}\|\Upsilon(\rho_\eps)-\Upsilon(\Theta)\|_{cb}=0.
\]
Since $\theta$ has Paulsen's similarity property with constant $s^{1/2}$, we may use Corollary \ref{C:paulsendirectsum} along with Theorem \ref{T:paulsenchar} to conclude that $\|\Upsilon(\Theta)\|_{cb}\leq s < r$. Thus, there is $\eps_0>0$ small enough so that $\|\Upsilon(\rho_{\eps_0})\|_{cb}<r$. Another application of Theorem \ref{T:paulsenchar} yields that $\rho_{\eps_0}$ has Paulsen's similarity property with constant $r^{1/2}$. Note that $\H_\omega$ can be identified with
\[
\H_\omega\oplus 0 \oplus 0 \oplus 0\subset \K_\omega^{(4)},
\]
in which case
\[
\omega(a)=P_{\H_\omega}\rho_{\eps_0}(a)|_{\H_\omega}, \quad a\in \M.
\]
But $\H_\omega$ is not reducing for $\rho_{\eps_0}(\M)$, since $\omega$ is non-zero. This contradicts the fact that $\omega$ is $\Pau_r(\M)$-extremal. We conclude that $\omega\in \Pau_r(\M,\H_\omega)\setminus \cup_{s<r}\Pau_r(\M,\H_\omega)$. 
Since $\omega$ is $\Pau_r(\M)$-extremal we know from Theorem \ref{T:Frextmult} that there is a unique homomorphism $\Omega: \A_\M\to B(\H_\omega)$  that agrees with $\omega$ on $\M$ and that has Paulsen's similarity property with constant $r^{1/2}$, but not with constant $s^{1/2}$ for any $s<r$. Using Theorem \ref{T:paulsenchar}, we conclude that $\|\Upsilon(\Omega)\|_{cb}=r$.

(2)  Let $\eps>0$ and define a completely bounded linear map $\delta:\M\to B(\H_\omega\oplus \H_\omega)$ as
\[
\delta(a)=
\begin{bmatrix}
\omega(a) & 0\\
\eps\omega(a) & 0
\end{bmatrix}, \quad a\in \M.
\]
Then, we have $\|\delta\|_{cb}=\sqrt{1+\eps^2}\|\omega\|_{cb}$. By Theorem \ref{T:paulsencbdilation2}, there is a Hilbert space $\K_\omega$ containing $\H_\omega\oplus \H_\omega$ and a homomorphism $\theta:C^*(\M)\to B(\K_\omega)$ such that
\[
\delta(a)=P_{\H_\omega\oplus \H_\omega}\theta(a)|_{\H_\omega\oplus \H_\omega}, \quad a\in\M.
\]
Moreover, we can assume that $\|\theta\|_{cb}\leq C$ where
\begin{align*}
C= \left(\left(\sqrt{1+\eps^2}\|\omega\|_{cb}\right)^{1/2}+\left(1+\sqrt{1+\eps^2}\|\omega\|_{cb}\right)^{1/2}\right)^2.
\end{align*}
By Theorem \ref{T:paulsensim}, we see that $\theta$ has Paulsen's similarity property with constant  $\sqrt{C}+\sqrt{1+C}$. Thus, 
\[
\delta\in\Pau_{(\sqrt{C}+\sqrt{1+C})^2}(\M,\H_\omega\oplus \H_\omega).
\]
On the other hand, note that $\H_\omega$ can be identified as the first coordinate in $\H_\omega\oplus \H_\omega$, in which case
\[
\omega(a)=P_{\H_\omega}\delta(a)|_{\H_\omega}, \quad a\in \M.
\]
The subspace $\H_\omega$ is clearly not reducing for $\delta(\M)$, since $\omega$ is non-zero. However, $\omega$ is assumed to be $\Pau_r(\M)$-extremal, so we must have that $(\sqrt{C}+\sqrt{1+C})^2>r$. Now, $C\leq 4 (1+\sqrt{1+\eps^2}\|\omega\|_{cb})$ so that
\[
r<4(5+4\sqrt{1+\eps^2}\|\omega\|_{cb}).
\]
Letting $\eps\to 0$ and rearranging yields the announced inequality.

\end{proof}

Using the previous theorem, we show that a $\CP_1(\M)$-extremal element, let alone a map merely similar to one, is not necessarily a $\Pau_r(\M)$-extremal if $r$ is allowed to be arbitrarily large.

\begin{example}\label{E:extsim1}
Let $\omega:B(\H)\to B(\H)$  be the identity map. Trivially, we see that $\omega$ has the unique extension property relative to $\Pau_r(B(\H))$ for every $r\geq 1$, and also relative to $\CP_1(B(\H))$. In particular, we see that $\omega$ is $\CP_1(B(\H))$-extremal by Lemma \ref{L:CP1uepext}. On the other hand, by virtue of part (2) of Theorem \ref{T:Frextbig}, we see that $\omega$ is not $\Pau_r(\M)$-extremal as soon as
\[
\frac{r}{16}-\frac{5}{4}>1
\]
since $\|\omega\|_{cb}=1$.

\end{example}

The previous example is a bit artificial. The following is a more satisfying example that does not rely on forcing $r$ to be large so that we can apply Theorem \ref{T:Frextbig}. We exhibit a linear map $\omega$ with $\|\omega\|_{cb}=r$ and with the property that there is an invertible operator $X$ such that $\|X\|=\|X^{-1}\|= r^{1/2}$ and such that the map 
\[
a\mapsto X\omega(a)X^{-1}, \quad a\in \M
\]
is  $\CP_1(\M)$-extremal, yet $\omega$ itself is not $\Pau_r(\M)$-extremal. The basic idea is to append a direct summand to create room.

\begin{example}\label{E:extsim2}
Let $\H$ be a Hilbert space and let $\eps>0$. Let $\H_\delta=\H\oplus \H$ and define a linear map $\delta:B(\H)\to B(\H_\delta)$ as 
\[
\delta(a)=
\begin{bmatrix}
a & 0\\
 \eps a & 0
\end{bmatrix}, \quad a\in B(\H).
\]
We see that $\|\delta\|_{cb}=\sqrt{1+\eps^2}$. Using Theorem \ref{T:paulsencbdilation2} and arguing as in the proof of part (2) of Theorem \ref{T:Frextbig}, we infer that there is $r\geq 1$ large enough so that $\delta\in \Pau_r(B(\H),\H_\delta)$. Compressing to the copy of $\H$ corresponding to the first coordinate in $\H_\delta$, we observe also that
\[
a=P_{\H}\delta(a)|_{\H}, \quad a\in B(\H).
\]
Next, fix an invertible operator $X\in B(\H)$ such that $\|X\|=\|X^{-1}\|=r^{1/2}\geq 1$. Let $\omega:B(\H)\to B(\H\oplus \H)$ be the homomorphism defined as
\[
\omega(a)=XaX^{-1}\oplus a, \quad a\in B(\H).
\]
Then, $\|\omega\|_{cb}=r$ and $\omega\in \Pau_r(B(\H),\H\oplus \H)$. Consider the linear map $\Delta:B(\H)\to B(\H\oplus \H_\delta)$ defined as
\[
\Delta(a)=XaX^{-1}\oplus \delta(a),\quad a\in B(\H).
\]
By Corollary \ref{C:paulsendirectsum}, we have that $\Delta\in \Pau_r(B(\H),\H\oplus\H_\delta)$ and clearly $\omega\prec \Delta$. Since the copy of $\H$ corresponding to the first coordinate in $\H_\delta$ is not reducing for $\delta(B(\H))$, the space $\H\oplus \H\subset \H\oplus \H_\delta$ is not reducing for $\Delta(B(\H))$. Thus, $\omega$ is not $\Pau_r(B(\H))$-extremal.

On the other hand, let $Y\in B(\H\oplus \H)$ be the invertible operator defined as $Y=X\oplus I$. Then, $\|Y\|=\|Y^{-1}\|\leq r^{1/2}$. Consider the map $\omega_Y:B(\H)\to B(\H\oplus \H)$ defined as
\[
\omega_Y(a)=Y^{-1}\omega(a)Y=a\oplus a, \quad a\in B(\H).
\]
Clearly, $\omega_Y$ has the unique extension property relative to $\CP_1(B(\H))$. Hence, $\omega_Y$ is $\CP_1(B(\H))$-extremal by Lemma \ref{L:CP1uepext}. Thus, $\omega$ is similar to a $\CP_1(B(\H))$-extremal element, yet it is not $\Pau_r(B(\H))$-extremal itself.
\end{example}

Finally, we exhibit an example that shows that given a linear map $\phi$, in general it is not sufficient that $\|\Upsilon(\phi)\|_{cb}\leq r$ to force $\phi$ to lie in $\Pau_r(\M)$. This was mentioned in passing at the beginning of Section \ref{S:Fr}.

\begin{example}\label{E:UCBneqFr}
We return to the map from Example \ref{E:double}. Let $\H=\bC$, let $\M=B(\H)$ and $\phi:\M\to B(\H)$ be defined as $\phi(\lambda)=2\lambda$ for each $\lambda\in \M$. Assume that $\phi\in \Pau_3(\M,\bC)$. Then, by Theorem \ref{T:Frext}, there is a $\Pau_3(\M)$-extremal element $\omega\in \Pau_3(\M, \H_\omega)$ with $\phi\prec \omega$. Since $\M$ is a $C^*$-algebra, we see from Theorem \ref{T:Frextmult} that $\omega(1)$ is an idempotent. Moreover, we have that 
\[
2=\phi(1)=P_{\H}\omega(1)|_{\H}
\]
which is easily seen to force $\|\omega(1)\|>2$. 
Next, we note that $\|\Upsilon(\omega)\|_{cb}=3$ by Theorem \ref{T:Frextbig}. Since $\M$ is unital, we have that $\Upsilon(\M)=\M\oplus \bC$ and
\[
2\|\omega(1)\|-1 \leq \|2\omega(1) - I_{\H_\omega}\| = \|\Upsilon(\omega)(1\oplus -1)\| \leq \|\Upsilon(\omega)\|_{cb} = 3
\]
whence  $\|\omega(1)\|\leq 2$, which is a contradiction. Thus, $\phi\notin \Pau_3(\M,\bC)$. 

On the other hand, a straightforward verification shows that
\[
\Upsilon(\phi)(\lambda\oplus \mu)=2\lambda-\mu, \quad \lambda\in\M,\mu\in \bC.
\]
Hence, $\|\Upsilon(\phi)\|_{cb}=\|\Upsilon(\phi)\|\leq 3$. 

\end{example}

\section{A scale of $C^*$-envelopes for representations of operator spaces}\label{S:C*env}

In this section, we define a scale of $C^*$-envelopes associated to operator spaces. Traditionally, one would expect such envelopes to depend only on the completely isometric isomorphism class of the operator space, and not on the choice of representation, as is the case in \cite{hamana1999},\cite{blecher2001},\cite{FHL2016}. However, since our present setting is that of possibly non-unital operator spaces, to obtain $C^*$-envelopes we will need to consider an operator space along with an associated completely isometric representation on some Hilbert space. We emphasize that unlike ours, the usual constructions (\cite{hamana1999},\cite{blecher2001},\cite{FHL2016}) do not require this additional piece of data, and although their resulting envelopes are not algebras, they exhibit stronger universality properties.
 
Let $\M$ be an operator space and put 
\[
\fd(\M)=(1+\aleph_0\dim \M)(\dim \M)^2.
\]
For each cardinal number $n\leq \fd(\M)$ we fix a Hilbert space $\fH_n$ of dimension $n$, and implicitly identify every other such space with $\fH_n$.  This is a standard procedure to avoid set theoretic difficulties, and this convention will be used tacitly henceforth.

Next, let $r\geq 1$ and let $\mu:\M\to B(\H_\mu)$ be a completely isometric linear map. For a cardinal number $\fc$ we let $\E_r(\mu,\fc)$ be the subset of elements $\omega\in\Pau_r(\mu(\M),\fH_{\fc})$ that are $\Pau_r(\mu(\M))$-extremal. Given $\omega\in \E_r(\mu,\fc)$, we denote by $\I_r(\omega)$ the collection of invertible operator $X\in B(\fH_\fc)$ such that $\|X\|=\|X^{-1}\|\leq r^{1/2}$ and with the property that the map $\omega_X:\mu(\M)\to B(\fH_\fc)$ defined as
\[
\omega_X(a)=X\omega(a)X^{-1}, \quad a\in \mu(\M)
\]
belongs to $\Pau_1(\mu(\M),\fH_\fc)$. We know from Theorem \ref{T:Frextsimuep} that $\I_r(\omega)$ is non-empty, and that for each $X\in \I_r(\omega)$ there is a unique $*$-homomorphism $\pi_{\omega,X}:C^*(\mu(\M))\to B(\fH_\fc)$ such that
\[
\pi_{\omega,X}(a)=X \omega(a)X^{-1}, \quad a\in \mu(\M).
\]
We let
\[
\fH^{r,\mu}=\bigoplus_{\fc\leq \fd(\M)}\bigoplus_{\omega\in \E_r(\mu,\fc)} \bigoplus_{X\in \I_r(\omega)}\fH_{\fc}
\]
and we define a $*$-homomorphism  $\eps_{\mu,r}: C^*(\mu(\M))\to B(\fH^{r,\mu})$ as
\[
\eps_{\mu,r}(a)=\bigoplus_{\fc\leq \fd(\M)}\bigoplus_{\omega\in \E_r(\mu,\fc)} \bigoplus_{X\in \I_r(\omega)}\pi_{\omega,X}(a), \quad a\in C^*(\mu(\M)).
\]
We verify that $\eps_{\mu,r}(\M)$ can be identified with $\mu(\M)$ in a meaningful way.

\begin{corollary}\label{C:embedding}
Let $\M$ be an operator space, let $\mu:\M\to B(\H_\mu)$ be a completely isometric linear map and let $r\geq 1$. Then, 
\[
r^{-1}\|a\|\leq \|\eps_{\mu,r}^{(n)}(a)\|\leq  \|a\|
\]
for every $n\in \bN$ and every $a\in M_n(\mu(\M))$.
\end{corollary}
\begin{proof}
It is clear that $\eps_{\mu,r}$ is completely contractive as it is a $*$-homomorphism. Furthermore,  we may use Corollary \ref{C:Frextnorm} to obtain a cardinal number $\fc$ such that $\fc\leq \fd(\M)$ and an element $\omega \in \Pau_r(\mu(\M),\fH_\fc)$ that is $\Pau_r(\mu(\M))$-extremal and such that for every $n\in \bN$ and every $a\in M_n(\mu(\M))$ we have $\|a\|\leq \|\omega^{(n)}(a)\|$. If $X\in \I_r(\omega)$, then 
\[
\|\eps_{\mu,r}^{(n)}(a)\|\geq \|X^{(n)}\omega^{(n)}(a)X^{-1(n)}\|\geq r^{-1}  \|a\|
\]
for every $n\in \bN$ and every $a\in M_n(\mu(\M))$.
\end{proof}

For $r\geq 1$, we define the \emph{$C_r^*$-envelope} of the pair $(\M,\mu)$ as
\[
C^*_{e,r}(\M,\mu)=\eps_{\mu,r}(C^*(\mu(\M))).
\]
If we let $\J_{\mu,r}=\ker \eps_{\mu,r}$, then we see that $C^*(\mu(\M))/\J_{\mu,r}\cong C^*_{e,r}(\M,\mu)$. In particular, if $C^*(\mu(\M))$ is simple, then the surjective $*$-homomorphism 
\[
\eps_{\mu,r}:C^*(\mu(\M))\to C^*_{e,r}(\M,\mu)
\]
is necessarily injective, so that $C_{e,r}^*(\M,\mu)$ and $C^*(\mu(\M))$ are $*$-isomorphic for every $r\geq 1$. In addition, 
if $\mu(\M)$ is a $C^*$-algebra to begin with, then in fact $\J_{\mu,r}$ is trivial and $\mu(\M)$ can be identified with the $C^*_r$-envelope of $(\M,\mu)$, as we show next.

\begin{corollary}\label{C:OAalg}
Let $\M$ be an operator space, let $\mu:\M\to B(\H_\mu)$ be a completely isometric linear map and let $r\geq 1$. If $\mu(\M)$ is  $C^*$-algebra, then $C_{e,r}^*(\M,\mu)$ is $*$-isomorphic to $C^*(\mu(\M))$.
\end{corollary}
\begin{proof}
By Corollary \ref{C:embedding}, we see that the surjective $*$-homomorphism 
\[
\eps_{\mu,r}:C^*(\mu(\M))\to C_{e,r}^*(\M,\mu)
\]
 is injective.
\end{proof}

Going back to the more interesting general case where $\mu(\M)$ is merely an operator space, we wish to establish that $C^*_{e,r}(\M,\mu)$ does not depend on the isomorphism class of $\M$ in an essential way. What kind of map should be considered an isomorphism in this context is an interesting question. A natural answer would be that an isomorphism should be an invertible map that preserves the class $\Pau_r$. Accordingly, we call a bijective linear map $\tau$ between two concretely represented operator spaces $\M$ and $\N$ a \emph{$\Pau$-isomorphism} if $\tau$ and $\tau^{-1}$ both lie in the class $\Pau_1$. 
If $\tau$ is a $\Pau$-isomorphism then for every $r\geq 1$ we see by virtue of Corollary \ref{C:Frcomp} that $\phi\circ \tau^{-1}\in \Pau_r(\N,\H_\phi)$ and $\psi\circ \tau\in \Pau_r(\M,\H_\psi)$ whenever $\phi\in \Pau_r(\M,\H_\phi)$ and $\psi\in \Pau_r(\N,\H_\psi)$.  

Let us exhibit two concretes instances of $\Pau$-isomorphisms. First, by Corollary \ref{C:hom}, we see that a bijective map $\tau:\M\to \N$ is a $\Pau$-isomorphism if it is assumed to extend to a completely isometric algebra isomorphism $\widehat\tau:\A_\M\to \A_\N$, where $\A_\M$ and $\A_\N$ denote the operator algebras generated by $\M$ and $\N$ respectively. Second, if $\M$ and $\N$ are both unital and $\tau:\M\to \N$ is a unital completely isometric linear isomorphism, then $\tau$ is necessarily a $\Pau$-isomorphism by Corollary \ref{C:F1}. 

Moreover, we have the following.

\begin{lemma}\label{L:Frisom}
Let $\M$ and $\N$ be operator spaces and let
\[
\mu:\M\to B(\H_\mu), \quad \nu:\N\to B(\H_\nu)
\]
be completely isometric linear maps. Let $r\geq 1$ and let $\fc$ be a cardinal number. Let $\tau:\mu(\M)\to \nu(\N)$ be a $\Pau$-isomorphism and let $\omega\in \E_r(\nu,\fc)$. Then $\omega\circ \tau \in \E_r(\mu,\fc)$ and $\I_r(\omega)=\I_r(\omega\circ \tau)$.
\end{lemma}
\begin{proof}
Let $\delta\in \Pau_r(\mu(\M),\H_\delta)$ such that $\omega\circ \tau\prec \delta$. Then, we have that $\delta\circ \tau^{-1}\in \Pau_r(\nu(\N),\H_\delta)$ and $\omega\prec\delta\circ \tau^{-1}$. Hence, the space $\fH_\fc$ is reducing for $\delta\circ \tau^{-1}(\nu(\N))=\delta(\mu(\M))$. This shows that $\omega\circ \tau$ is $\Pau_r(\mu(\M))$-extremal.
Next, let $X\in \I(\omega)$, so that $\omega_X\in \Pau_1(\nu(\N),\fH_\fc)$. Since $\tau$ is a $\Pau$-isomorphism, we also have $\omega_X\circ \tau\in \Pau_1(\mu(\M),\fH_\fc)$. Hence, $X\in \I_r(\omega\circ \tau)$. The same argument can be used to establish the reverse inclusion, by symmetry.
\end{proof}

We now prove that the $C_r^*$-envelope is invariant under $\Pau$-isomorphisms.

\begin{theorem}\label{T:C*envinv}
Let $\M$ and $\N$ be operator spaces and let
\[
\mu:\M\to B(\H_\mu), \quad \nu:\N\to B(\H_\nu)
\]
be completely isometric linear maps. Let $r\geq 1$ and let $\tau:\mu(\M)\to \nu(\N)$ be a $\Pau$-isomorphism. Then, the maps $\eps_{\mu,r}\circ \tau^{-1}$ and $\eps_{\nu,r}$ are unitarily equivalent on $\nu(\N)$. In particular, $C_{e,r}^*(\M,\mu)$ and $C_{e,r}^*(\N,\nu)$ are unitarily equivalent.
\end{theorem}
\begin{proof}
We first note that $\fd(\M)=\fd(\N)$. By Lemma \ref{L:Frisom}, we see that $\zeta\in \E_r(\nu,\fc)$  if and only if $\zeta=\omega\circ \tau^{-1}$ for some $\omega\in \E_r(\mu,\fc)$. Moreover, we have that $\I_r(\omega)=\I_r(\omega\circ \tau^{-1})$ for every $\omega\in \E_r(\mu,\fc)$. For $b\in \nu(\N)$ we see that
\[
\eps_{\mu,r}\circ \tau^{-1}(b)=\bigoplus_{\fc\leq \fd(\M)}\bigoplus_{\omega\in \E_r(\mu,\fc)} \bigoplus_{X\in \I_r(\omega)}X(\omega\circ \tau^{-1}(b))X^{-1}
\]
which, after a unitary permutation of the coordinates of $\fH^{r,\mu}$, becomes
\[
\bigoplus_{\fc\leq \fd(\M)}\bigoplus_{\zeta\in \E_r(\nu,\fc)} \bigoplus_{X\in \I_r(\zeta)}X\zeta(b)X^{-1}=\eps_{\nu,r}(b).
\]
Hence, there is a unitary operator $U:\fH^{r,\nu}\to \fH^{r,\mu}$ such that 
\[
\eps_{\mu,r}\circ \tau^{-1}(b)=U\eps_{\nu,r}(b) U^*, \quad b\in \nu(\N).
\]
Consequently,
\begin{align*}
C^*_{e,r}(\M,\mu)&=C^* \left(\eps_{\mu,r}(\mu(\M))\right)=C^*\left(\eps_{\mu,r}\circ \tau^{-1}(\nu(\N))\right)\\
&=UC^*\left( \eps_{\nu,r}(\nu(\N))\right)U^*=U C^*_{e,r}(\N,\nu)U^*.
\end{align*}
\end{proof}

Now, the reader may wonder exactly how the construction of the $C^*_r$-envelope depends on the completely isometric representation $\mu$. More precisely, if 
\[
\mu_1:\M\to B(\H_1), \quad \mu_2:\M\to B(\H_2)
\]
are completely isometric linear maps, it is not clear at the moment what relation exists, if any, between $C^*_{e,r}(\M,\mu_1)$ and $C^*_{e,r}(\M,\mu_2)$. Indeed, Theorem \ref{T:C*envinv} only guarantees invariance of the envelope under the assumption that $\mu_2\circ \mu_1^{-1}$ is a $\Pau$-isomorphism. The following example shows that in general the $C^*_r$-envelopes can differ. We are grateful to Ken Davidson for suggesting the idea therein.

\begin{example}\label{E:C^*repinv}
Let $\M=\bC$ and let $r\geq 1$. Let $\mu_1:\M\to \bC$ be the identity map and let
\[
\mu_2:\M\to M_2(\bC)
\]
be the completely isometric linear map defined as
\[
\mu_2(\lambda)=\begin{bmatrix}
0 & \lambda\\
0 & 0
\end{bmatrix}, \quad \lambda\in \bC.
\]
Now, $\mu_1(\M)=\bC$ is a $C^*$-algebra, whence $C^*_{e,r}(\M,\mu_1)\cong \bC$ by Corollary \ref{C:OAalg} . On the other hand, it is easily verified that $C^*(\mu_2(\M))=M_2(\bC)$ and so it is simple. In that case, we saw previously that $C^*_{e,r}(\M,\mu_2)\cong C^*(\mu_2(\M))=M_2(\bC)$. Thus, $C^*_{e,r}(\M,\mu_2)$ is not $*$-isomorphic to $C^*_{e,r}(\M,\mu_1)$.

\end{example}

Nevertheless, we may use Theorem \ref{T:C*envinv} to establish the following universality property, which justifies the terminology ``envelope".

\begin{corollary}\label{C:C*envuniv}
Let $\M$ and $\N$ be operator spaces and let
\[
\mu:\M\to B(\H_\mu), \quad \nu:\N\to B(\H_\nu)
\]
be completely isometric linear maps. Let $r\geq 1$  and let $\tau:\mu(\M)\to \nu(\N)$ be a $\Pau$-isomorphism.  Then, there is a surjective $*$-homomorphism 
\[
\rho:C^*(\nu(\N))\to C^*_{e,r}(\M,\mu)
\]
such that $\rho\circ \tau=\eps_{\mu,r}$ on $\mu(\M)$.
\end{corollary}
\begin{proof}
By Theorem \ref{T:C*envinv}, there is a unitary operator $U:\fH^{r,\nu}\to\fH^{r,\mu}$ such that 
\[
U C^*_{e,r}(\N,\nu)U^*=C^*_{e,r}(\M,\mu)
\]
and
\[
\eps_{\mu,r}\circ \tau^{-1}(b)=U\eps_{\nu,r}(b) U^*, \quad b\in \nu(\N).
\]
Define $\rho:C^*(\nu(\N))\to C_{e,r}^*(\M,\mu)$ as 
\[
\rho(b)=U \eps_{\nu,r}(b)U^*, \quad b\in C^*(\nu(\N)).
\]
Then, $\rho$ is a surjective $*$-homomorphism and $\rho\circ \tau=\eps_{\mu,r}$ on $\mu(\M)$.
\end{proof}

In light of the preceding developments, it seems relevant here to point out that the embedding $\eps_{\mu,r}:\mu(\M)\to \eps_{\mu,r}(\mu(\M))$ is not known to be a $\Pau$-isomorphism. We will revisit this issue below in Theorem \ref{T:conjequiv}.

Next, we wish to relate the $C^*$-algebra $C_{e,1}^*(\M,\mu)$ and the usual (unital) $C^*$-envelope of $\mu(\M)$. For that purpose, we introduce some terminology. Let $\M$ be a unital operator space and let $\H$ be a Hilbert space.  We denote by $\UCB_1(\M, \H)$ the set of unital completely contractive linear maps $\phi:\M\to B(\H)$.  
Let $\K$ be a Hilbert space containing $\H$ and let $\delta\in \UCB_1(\M,\K)$. If $\delta$ satisfies
\[
\phi(a)=P_{\H}\delta(a)|_{\H}, \quad a\in \M
\]
then we write $\phi\prec \delta$. An element $\omega\in \UCB_1(\M,\H_\omega)$ is said to be \emph{$\UCB_1(\M)$-extremal} if whenever $\delta\in \UCB_1(\M,\H_\delta)$ satisfies $\omega\prec \delta$ then we must have that $\H_\omega$ is reducing for $\delta(\M)$. We now relate two kinds of extremality.

\begin{lemma}\label{L:UCB}
Let $\M$ be an operator space and let $\omega\in \CP_1(\M,\H_\omega)$. Then  $\omega$ is $\CP_1(\M)$-extremal if and only if $\Upsilon(\omega)$ is $\UCB_1(\Upsilon(\M))$-extremal.
\end{lemma}
\begin{proof}
Assume that $\Upsilon(\omega)$ is $\UCB_1(\Upsilon(\M))$-extremal. Let $\delta\in \CP_1(\M,\H_\delta)$ such that $\omega\prec \delta$. Then,  $\Upsilon(\delta)\in \UCB_1(\Upsilon(\M),\H_\delta)$ and it is easily verified that  $\Upsilon(\omega)\prec \Upsilon(\delta)$. We conclude that $\H_\omega$ is reducing for $\Upsilon(\delta)(\Upsilon(\M))$, and in particular for $\delta(\M)$. Hence $\omega$ is $\CP_1(\M)$-extremal. 

Conversely, assume that $\omega$ is $\CP_1(\M)$-extremal.  Let $\Delta\in \UCB_1(\Upsilon(\M),\H_\Delta)$ such that $\Upsilon(\omega)\prec \Delta$. Recall now that the map $\upsilon:\M\to \Upsilon(\M)$ is the restriction of a $*$-homomorphism, so that $\Delta\circ \upsilon\in \CP_1(\M,\H_\Delta)$. We note that $\omega\prec \Delta\circ \upsilon$.  We conclude that $\H_\omega$ is reducing for $\Delta( \upsilon(\M))$. Since $\Delta$ is unital, in fact $\H_\omega$ must  be reducing for $\Delta(\Upsilon(\M))$, so that $\Upsilon(\omega)$ is $\UCB_1(\M)$-extremal. 
\end{proof}

We now recall the construction of the standard unital $C^*$-envelope for a (concrete) unital operator space $\M$. If $\fc$ is a cardinal number, then we let $\B(\M,\fc)$ denote the collection of elements $\beta\in \UCB_1(\M,\fH_\fc)$ that are $\UCB_1(\M)$-extremal. It is well-known that every $\beta\in \B(\M,\fH_\fc)$ extends to a unique unital $*$-homomorphism $\pi_\beta:C^*(\M)\to B(\fH_\fc)$. Then, the  \emph{unital $C^*$-envelope} of $\M$ can be defined as 
\[
UC_e^*(\M)=\kappa(C^*(\M))
\]
where $\kappa$ is the unital $*$-homomorphism defined as
\[
\kappa=\bigoplus_{\fc\leq \fd(\M)}\bigoplus_{\beta\in \B(\M,\fc)} \pi_\beta
\]
that is completely isometric on $\M$. 
The reader should consult \cite{arveson1969},\cite{dritschel2005},\cite{arvesonUEP},\cite{davidson2015} for more details. Note that the usual name for $UC_e^*(\M)$ is simply the $C^*$-envelope and the standard notation is $C_e^*(\M)$, but in the context of this paper we want to emphasize the fact that the maps involved in the previous construction are all unital. 

We can now relate the unital $C^*$-envelope to the $C^*_1$-envelope.

\begin{theorem}\label{T:C*env}
Let $\M$ be an operator space and let $\mu:\M\to B(\H_\mu)$ be a completely isometric linear map. Then, $C_{e,1}^*(\M,\mu)+\bC I_{\fH^{1,\mu}}$ is $*$-isomorphic to $UC_e^*(\Upsilon(\mu(\M)))$.
\end{theorem}
\begin{proof}
By Subsection \ref{SS:unitization}, Corollary \ref{C:F1} and Lemma \ref{L:UCB}, we see that the $\UCB_1(\Upsilon(\mu(\M)))$-extremal elements are precisely those of the form $\Upsilon(\omega)$ for some $\Pau_1(\mu(\M))$-extremal element $\omega$. 
Hence, we find that
\[
\kappa=\bigoplus_{\fc\leq \fd(\Upsilon(\mu(\M)))}\bigoplus_{\beta\in \B(\Upsilon(\mu(\M)),\fc)} \pi_\beta
\]
is unitarily equivalent to
\[
\bigoplus_{\fc\leq \fd(\Upsilon(\mu(\M)))}\bigoplus_{\omega\in \E_1(\mu,\fc)} \pi_{\Upsilon(\omega)}.
\]
Since $\aleph_0 \dim \M=\aleph_0 \dim \Upsilon(\mu(\M))$ we conclude that $\fd(\Upsilon(\mu(\M)))=\fd(\M)$ and thus
\begin{align*}
UC_e^*(\Upsilon(\mu(\M)))&=C^*\left(\left\{\bigoplus_{\fc\leq \fd(\Upsilon(\mu(\M)))}\bigoplus_{\beta\in \B(\Upsilon(\mu(\M)),\fc)}\beta(b):b\in \Upsilon(\mu(\M))\right\}\right)
\end{align*}
is unitarily equivalent to
\begin{align*}
&C^*\left(\left\{\bigoplus_{\fc\leq  \fd(\M)}\bigoplus_{\omega\in \E_1(\mu,\fc)}\Upsilon(\omega)(b):b\in \Upsilon(\mu(\M))\right\}\right).
\end{align*}
Since $\I_1(\omega)$ consists of unitary operators for every $\omega\in \E_1(\mu,\fc)$, we see that $UC_e^*(\Upsilon(\mu(\M)))$ is $*$-isomorphic to 
\begin{align*}
&C^*\left(\left\{\bigoplus_{\fc\leq  \fd(\M)}\bigoplus_{\omega\in \E_1(\mu,\fc)}\bigoplus_{V\in \I_1(\omega)}V\Upsilon(\omega)(b)V^*:b\in \Upsilon(\mu(\M))\right\}\right)\\
&=\bC I_{\fH^{1,\mu}}+C^*\left(\left\{\bigoplus_{\fc\leq  \fd(\M)}\bigoplus_{\omega\in \E_1(\mu,\fc)}\bigoplus_{V\in \I_1(\omega)}V\omega(a)V^*:a\in \mu(\M)\right\}\right)\\
&=\bC I_{\fH^{1,\mu}}+C^*\left(\left\{\bigoplus_{\fc\leq  \fd(\M)}\bigoplus_{\omega\in \E_1(\mu,\fc)}\bigoplus_{V\in \I_1(\omega)}\pi_{\omega,V}(a):a\in \mu(\M)\right\}\right)\\
&=\bC I_{\fH^{1,\mu}}+C^*_{e,1}(\M,\mu).
\end{align*}

\end{proof}

The outstanding issue that remains to be addressed is the relationship, for a given $\mu$, between the $C^*_{r}$-envelope and the $C^*_{s}$-envelope for $s\geq r$. We elucidate it partially with the following.

\begin{theorem}\label{T:OAscale}
Let $\M$ be an operator space, let $\mu:\M\to B(\H_\mu)$ be a completely isometric linear map and let $s\geq r\geq 1$. Then, 
\[
\eps_{\mu,r}(a)=Y_r \gamma_{s,r}( Y_s^{-1}\eps_{\mu,s}(a) Y_s)Y^{-1}_r, \quad a\in \mu(\M)
\]
where $Y_r\in B(\fH^{r,\mu})$ and $Y_s\in B(\fH^{s,\mu})$ are invertible operators with 
\[
\|Y_r\|=\|Y_r^{-1}\|\leq r^{1/2}, \quad \|Y_s\|=\|Y_s^{-1}\|\leq s^{1/2}
\]
and $\gamma_{s,r}$ is a completely contractive map. 
\end{theorem}

\begin{proof}
For notational convenience, we define $\iota_{\mu,r}:\mu(\M)\to B(\fH^{r,\mu})$ as
\[
\iota_{\mu,r}(a)=\bigoplus_{\fc\leq \fd(\M)}\bigoplus_{\omega\in \E_r(\mu,\fc)} \bigoplus_{X\in \I_r(\omega)}\omega(a), \quad a\in \mu(\M).
\]
 Let $\fc\leq \fd(\M)$ and let $\omega\in \E_r(\mu,\fc)$. Then $\omega \in \Pau_s(\mu(\M),\fH_\fc)$. By virtue of Theorem \ref{T:Frext}, there is a Hilbert space $\K_{\omega}$ containing $\fH_\fc$ with
\[
\dim \K_\omega\leq (1+\aleph_0\dim \M)\fc\leq  (1+\aleph_0\dim \M)\fd(\M)=\fd(\M)
\]
along with a linear map $\zeta_\omega\in \Pau_s(\mu(\M),\K_{\omega})$ that is $\Pau_s(\mu(\M))$-extremal and such that
\[
\omega(a)=P_{\fH_\fc}\zeta_\omega(a)|_{\fH_\fc}, \quad a\in \mu(\M).
\]
There is a cardinal number $\fc_\omega \leq \fd(\M)$ and a unitary operator $U_\omega:\K_{\omega}\to \fH_{\fc_\omega}$ such that  $\zeta'_\omega\in \E_s(\mu, \fH_{\fc_\omega})$, where
\[
\zeta'_\omega(a)=U_\omega \zeta_\omega(a) U_\omega^*,\quad a\in \mu(\M).
\] 
We note that 
\[
\cup_{\fc\leq \fd(\M)}\{\zeta'_\omega:\omega\in \E_r(\mu,\fc)\}\subset \cup_{\fc\leq \fd(\M)}\E_s(\mu,\fc).
\]
We can then define a completely contractive surjective linear map
\[
\gamma_{s,r}:\iota_{\mu,s}(\mu(\M))\to \iota_{\mu,r}(\mu(\M))
\]
as
\[
\gamma_{s,r}\left( \bigoplus_{\fc\leq \fd(\M)}\bigoplus_{\zeta\in \E_s(\mu,\fc)}\bigoplus_{X\in \I_s(\zeta)}\zeta(a)\right)=
\bigoplus_{\fc\leq \fd(\M)}\bigoplus_{\omega\in \E_r(\mu,\fc)} \bigoplus_{X\in \I_r(\omega)}P_{\fH_{\fc}}U^*_\omega\zeta'_\omega(a)U_\omega|_{\fH_{\fc}}
\]
for every $a\in \mu(\M)$. Note then that $\gamma_{s,r} \circ \iota_{\mu,s}=\iota_{\mu,r}$ on $\mu(\M)$.
Put
\[
Y_r=\bigoplus_{\fc\leq \fd(\M)}\bigoplus_{\omega\in \E_r(\mu,\fc)}\bigoplus_{X\in \I_r(\omega)} X
\]
\[
Y_s=\bigoplus_{\fc\leq \fd(\M)}\bigoplus_{\zeta\in \E_s(\mu,\fc)}\bigoplus_{X\in \I_s(\zeta)}X
\] 
which satisfy 
\[
\|Y_r\|=\|Y_r^{-1}\|\leq r^{1/2}, \quad \|Y_s\|=\|Y_s^{-1}\|\leq s^{1/2}.
\]
For $a\in \mu(\M)$ we have
\[
Y_r\iota_{\mu,r}(a)Y_r^{-1}=\eps_{\mu,r}(a), \quad Y_s\iota_{\mu,s}(a)Y_s^{-1}=\eps_{\mu,s}(a)
\]
whence
\begin{align*}
\eps_{\mu,r}(a)&=Y_r\iota_{\mu,r}(a)Y_r^{-1}=Y_r \gamma_{s,r}(\iota_{\mu,s}(a))Y_r^{-1}\\
&=Y_r \gamma_{s,r}(Y_s^{-1}\eps_{\mu,s}(a)Y_s)Y_r^{-1}.
\end{align*}
\end{proof}

The relationship between the $C^*_1$-envelope and the $C^*_r$-envelope is clearer, as the next result shows. 

\begin{theorem}\label{T:C1Cr}
Let $\M$ be an operator space, let $\mu:\M\to B(\H_\mu)$ be a completely isometric linear map and let $r\geq 1$. Then, there is a surjective $*$-homomorphism $\Gamma_{r}:C^*_{e,1}(\M,\mu)\to C^*_{e,r}(\M,\mu)$ such that $\Gamma_{r}\circ \eps_{\mu,1}=\eps_{\mu,r}$.
\end{theorem}
\begin{proof}
Let $\fc$ be a cardinal number satisfying $\fc\leq \fd(\M)$, let $\omega\in \E_r(\mu,\fc)$ and let $X\in \I_r(\omega)$. Then, we have $\omega_X\in \E_1(\mu,\fc)$ by Corollary \ref{C:Frextsimbdry}. Moreover, we note that if $\zeta\in \E_1(\mu,\fc)$, then $\I_1(\zeta)$ coincides with the set of unitary operators on $\fH_\fc$, so in particular it contains the identity $I$. Thus, we see that there is an element  $\zeta(\omega,X)\in \E_1(\mu,\fc)$ such that 
\[
\pi_{\zeta(\omega,X),I}=\pi_{\omega,X}.
\]
Put
\[
\Pi=\left\{ \zeta(\omega,X):\omega\in \E_r(\mu,\fc),X\in \I_r(\omega)\right\}\subset \E_1(\mu,\fc).
\]
We define a map $\gamma_r$ on $C^*_{e,1}(\M,\mu)$ to be the corresponding compression, so that
\[
\gamma_{r}\left( \bigoplus_{\fc\leq \fd(\M)}\bigoplus_{\zeta\in \E_1(\mu,\fc)}\bigoplus_{U\in \I_1(\zeta)}\pi_{\zeta,U}(a)\right)=\bigoplus_{\fc\leq \fd(\M)}\bigoplus_{\zeta\in \Pi}\pi_{\zeta,I}(a)
\]
for every $a\in C^*(\mu(\M))$. Clearly, $\gamma_r$ is a $*$-homomorphism. By choice of $\Pi$, there is a unitary operator $V$ such that
\[
V^*\gamma_{r}\left( \bigoplus_{\fc\leq \fd(\M)}\bigoplus_{\zeta\in \E_1(\mu,\fc)}\bigoplus_{U\in \I_1(\zeta)}\pi_{\zeta,U}(a)\right)V=
\bigoplus_{\fc\leq \fd(\M)}\bigoplus_{\omega\in \E_r(\mu,\fc)}\bigoplus_{X\in \I_r(\omega)}\pi_{\omega,X}(a)
\]
for every $a\in C^*(\mu(\M))$.
Define $\Gamma_r:C_{e,1}^*(\M,\mu)\to C_{e,r}^*(\M,\mu)$ as 
\[
\Gamma_r(b)=V^*\gamma_r(b)V, \quad b\in  C_{e,1}^*(\M,\mu).
\]
 Then, we see that $\Gamma_r$ is a $*$-homomorphism and that $\Gamma_{r}\circ \eps_{\mu,1}=\eps_{\mu,r}$. In particular $\Gamma_r$ is surjective.

\end{proof}

The question arises whether there is a $*$-isomorphism 
\[
\pi:C_{e,r}^*(\M,\mu)\to C_{e,1}^*(\M,\mu)
\]
satisfying $\pi\circ \eps_{\mu,r}=\eps_{\mu,1}$ on $\mu(\M)$. It is not immediately clear what the answer should be, especially in view of Examples \ref{E:extsim1} and \ref{E:extsim2}. Nevertheless, we suspect that this $*$-isomorphism does exist, although we cannot prove it in general. We can however establish the following related fact.

\begin{theorem}\label{T:conjequiv}
The following statements hold.
\begin{enumerate}

\item[\rm{(1)}] Let $\M$ be an operator space, let $\mu:\M\to B(\H_\mu)$ be a completely isometric linear map and let $r\geq 1$. Assume that the map $\eps_{\mu,r}:\mu(\M)\to \eps_{\mu,r}(\mu(\M))$ is a $\Pau$-isomorphism. Then, there is a $*$-isomorphism $\pi:C_{e,r}^*(\M,\mu)\to C^*_{e,1}(\M,\mu)$ satisfying $\pi\circ \eps_{\mu,r}=\eps_{\mu,1}$ on $\mu(\M)$. 

\item[\rm{(2)}] Let $\A$ be an operator algebra, let $\alpha:\A\to B(\H_\alpha)$ be a completely isometric homomorphism and let $r\geq 1$. Assume that there is a $*$-isomorphism $\pi:C_{e,r}^*(\A,\alpha)\to C^*_{e,1}(\A,\alpha)$ satisfying $\pi\circ \eps_{\alpha,r}=\eps_{\alpha,1}$ on $\alpha(\A)$. Then, the map $\eps_{\alpha,r}:\alpha(\A)\to \eps_{\alpha,r}(\alpha(\A))$ is a $\Pau$-isomorphism.
\end{enumerate}
\end{theorem}
\begin{proof}
(1) Recall that $C^*_{e,r}(\M,\mu)=C^*(\eps_{\mu,r}(\mu(\M)))$. Since $\eps_{\mu,r}:\mu(\M)\to \eps_{\mu,r}(\mu(\M))$ is a $\Pau$-isomorphism, we may apply  Corollary \ref{C:C*envuniv} to obtain a surjective $*$-homomorphism $\rho:C^*_{e,r}(\M,\mu)\to C_{e,1}^*(\M,\mu)$ such that $\rho\circ \eps_{\mu,r}=\eps_{\mu,1}$ on $\mu(\M)$. On the other hand, by Theorem \ref{T:C1Cr} there is a surjective $*$-homomorphism $\Gamma_r: C^*_{e,1}(\M,\mu)\to C_{e,r}^*(\M,\mu)$ such that $\Gamma_r\circ \eps_{\mu,1}=\eps_{\mu,r}$ on $\mu(\M)$. We now see that
\[
(\Gamma_r\circ \rho)(\eps_{\mu,r}(a))=\eps_{\mu,r}(a), \quad a\in \mu(\M).
\]
We conclude that
\[
\Gamma_r\circ \rho(a)=a, \quad a\in C^*_{e,r}(\M,\mu).
\]
Hence, $\rho$ is injective as desired.

(2) By Corollary \ref{C:embedding}, we see that $\eps_{\alpha,1}:\alpha(\A)\to \eps_{\alpha,1}(\alpha(\A))$ is a completely isometric algebra isomorphism, and thus $\eps_{\alpha,1}$ is a $\Pau$-isomorphism. On the other hand, we see that $\eps_{\alpha,r}=\pi^{-1}\circ \eps_{\alpha,1}$ on $\alpha(\A)$. Since $\pi$ is a $*$-isomorphism, we infer that $\eps_{\alpha,r}:\alpha(\A)\to \eps_{\alpha,r}(\alpha(\A))$ is a $\Pau$-isomorphism as well.\end{proof}

Thus, for an operator algebra $\A$ and a completely isometric homomorphism $\alpha$  on it, whether the map $\eps_{\alpha,r}:\alpha(\A)\to \eps_{\alpha,r}(\alpha(\A))$ is a $\Pau$-isomorphism is equivalent to the existence of a $*$-isomorphism $\pi:C_{e,r}^*(\A,\alpha)\to C_{e,1}^*(\A,\alpha)$ satisfying some additional natural condition. In general, we make the following conjecture.

\vspace{.5cm}

\textbf{Conjecture. }Let $\M$ be an operator space, let $\mu:\M\to B(\H_\mu)$ be a completely isometric linear map and let $r\geq 1$. Then, the $C^*$-algebras $C_{e,1}^*(\M,\mu)$ and $C_{e,r}^*(\M,\mu)$ are $*$-isomorphic.

\vspace{.5cm}

We close the paper by verifying this conjecture in some cases of interest. We already observed that if $C^*(\mu(\M))$ is simple, then $C_{e,r}^*(\M,\mu)$ and $C^*(\mu(\M))$ are $*$-isomorphic for every $r\geq 1$. On the other hand, if we assume that $C^*(\mu(\M))$ contains the ideal $\K$ of compact operators on $\H_\mu$, then we can also achieve some partial success. 
The following can be viewed as a variation of an important fact called \emph{Arveson's boundary theorem} \cite{arveson1972}. 

\begin{theorem}\label{T:Crbdrythm}
Let $\M$ be an operator space and let $\mu:\M\to B(\H_\mu)$ be a completely isometric linear map such that $C^*(\mu(\M))$ contains the ideal $\K$ of compact operators on $\H_\mu$. Let $q:C^*(\mu(\M))\to C^*(\mu(\M))/\K$ denote the quotient map. Let $r\geq 1$ and assume that there is an integer $n\in \bN$ and an element $a\in M_n(\mu(\M))$ such that $\|q^{(n)}(a)\|<r^{-1}\|a\|$. Then, the map $\eps_{\mu,r}:C^*(\mu(\M))\to C_{e,r}^*(\M,\mu)$ is a $*$-isomorphism.
\end{theorem}
\begin{proof}
Basic representation theory for $C^*$-algebras \cite{dixmier1977} stipulates that the $*$-homo\-morphism 
\[
\eps_{\mu,r}:C^*(\mu(\M))\to B(\fH^{r,\mu})
\]
is unitarily equivalent to 
\[
\text{id}^{(\fc)}\oplus \sigma\circ q
\]
for some cardinal number $\fc$ and some $*$-homomorphism $\sigma$ on $C^*(\mu(\M))/\K$. By assumption, we see that $\|(\sigma\circ q)^{(n)}(a)\|<r^{-1}\|a\|$. In view of Corollary \ref{C:embedding}, we conclude that $\fc\neq 0$. In particular $\eps_{\mu,r}$ is injective and thus 
\[
\eps_{\mu,r}:C^*(\mu(\M))\to C_{e,r}^*(\M,\mu)
\]
is a $*$-isomorphism.
\end{proof}

Many important classical examples of operator spaces fit into the framework of the previous theorem, such as the higher-dimensional Toeplitz algebra of multivariate operator theory.

\begin{example}\label{E:Ad}
Fix a positive integer $d\geq 2$ and let $\bB_d\subset \bC^d$ denote the open unit ball. The \emph{Drury-Arveson space} $H^2_d$ is the reproducing kernel Hilbert space on $\bB_d$ with reproducing kernel given by the formula
\[
k(z,w)=\frac{1}{1-\langle z,w\rangle_{\bC^d}}, \quad z,w\in \bB_d.
\] 
This is a Hilbert space of holomorphic functions on $\bB_d$, and it is a very natural higher-dimensional analogue of the classical Hardy space on the unit disc. It can be identified with the symmetric Fock space over $\bC^d$ \cite{arveson1998}, \cite{davidson1998}. 

A function $\phi:\bB_d\to \bC$ is a \emph{multiplier} for $H^2_d$ if $\phi f\in H^2_d$ for every $f\in H^2_d$. Examples of such functions include the holomorphic polynomials in $d$ variables.  Now, every multiplier $\phi$ gives rise to a bounded linear multiplication operator $M_\phi\in B(H^2_d)$, and the identification $\phi\mapsto M_\phi$ allows us to view the multiplier algebra as an operator algebra on $H^2_d$. The book \cite{agler2002} is an excellent reference on these topics.

Let $\A_d\subset B(H^2_d)$ denote the norm closure of the polynomial multipliers. This algebra is of tremendous importance in multivariate operator theory \cite{arveson1998},\cite{popescu1991},\cite{popescu2006}, \cite{CD2016abscont} and is also the target of intense research in function theory \cite{costea2011},\cite{DRS2015},\cite{CD2016duality}. One of the distinguishing features of $\A_d$ is that it is not a uniform algebra, in the sense that the norm of a multiplier does not coincide with its supremum norm over $\bB_d$. 

In fact, something more precise is known to hold. The so-called \emph{Toeplitz algebra} $\fT_d=C^*(\A_d)$ contains the ideal $\K$ of compact operators on $H^2_d$, and the quotient $\fT_d/\K$ is $*$-isomorphic to the $C^*$-algebra of continuous functions on the unit sphere \cite[Theorem 5.7]{arveson1998}. Moreover, since $d\geq 2$ the quotient map $q:\fT_d\to \fT_d/\K$ is not bounded below on $\A_d$ (see \cite[Theorem 3.3]{arveson1998} or \cite[Theorem 2.4]{davidson1998}). In particular, we see that $C_{e,r}^*(\A_d, \text{id})$ is $*$-isomorphic to $\fT_d$ for every $r\geq 1$, by Theorem \ref{T:Crbdrythm}.

\end{example}

Crucial to the preceding example was the fact that quotient map $q:\fT_d\to \fT_d/\K$ is not bounded below on $\A_d$, which fails when $d=1$. The algebra $\A_1$ is simply the usual disc algebra consisting of  the holomorphic functions on the open unit disc which extend to be continuous on the closed disc.  The identification of the $C^*_r$-envelope of $\A_1$ requires some classical uniform algebra machinery. In fact, we obtain the following more general result.

\begin{theorem}\label{T:Crcomm}
Let $X$ be a compact metrizable space and let $\A\subset C(X)$ be a closed unital subalgebra which separates the points of $X$. Let $\Sigma_\A\subset X$ denote the Shilov boundary of $\A$. Let $\alpha:\A\to B(\H_\alpha)$ be a completely isometric algebra homomorphism.  Then, for each $r\geq 1$ the $C^*$-algebras $C^*_{e,r}(\A,\alpha)$ and $C(\Sigma_\A)$ are $*$-isomorphic.
\end{theorem}
\begin{proof}
Let $\sigma:\A\to C(\Sigma_\A)$ be the completely isometric algebra homomorphism given by restriction. Since completely isometric algebra isomorphisms are $\Pau$-isomorphisms, by virtue of Theorem \ref{T:C*envinv} we see that it suffices to show that $C^*_{e,r}(\A,\sigma)$ and $C(\Sigma_\A)$ are $*$-isomorphic. 

It follows from the Stone-Weierstrass theorem that $C^*(\sigma(A))=C(\Sigma_\A)$. Now, we note that $\eps_{\sigma,r}(1)$ is a self-adjoint projection, so that there is a unital commutative $C^*$-algebra $\fA$ such that 
\[
C_{e,r}^*(\A,\sigma)=\eps_{\sigma,r}(C(\Sigma_\A))=\fA\oplus \{0\}.
\]
The maximal ideal space of $\fA$ can be identified with a compact subset $Y\subset \Sigma_\A$. If we denote by $g$ the Gelfand transform of $\fA$, then we know that $g:\fA\to C(Y)$ is a $*$-isomorphism by the Gelfand-Naimark theorem. Moreover, note that
\[
g\circ \eps_{\sigma,r}(f)=f|_Y, \quad f\in C(\Sigma_\A).
\]
Let $\xi\in \Sigma_\A$ be a peak point for $\A$, so that there is $f\in \A$ such that $f(\xi)=1$ and $|f(x)|<1$ if $x\in \Sigma_\A\setminus\{\xi\}$. Assume that $\xi\notin Y$. Since $Y$ is compact, there is a natural number $n$ large enough so that $\|f^n|_Y\|_{C(Y)}<1/r$. Then, we see that $\|g\circ \eps_{\sigma,r}(f^n)\|<1/r$ which implies that $\|\eps_{\sigma,r}(f^n)\|<1/r$. But since $\|f^n\|=1$ and $f^n\in \A$, this contradicts Corollary \ref{C:embedding}. We thus conclude that $Y$ contains all peak points for $\A$. Since $Y$ is closed, this implies that $Y=\Sigma_\A$ (see the discussion following \cite[Theorem 11.6]{gamelin1969} or \cite[Chapter 8]{phelps2001} for instance). Thus $\fA$ is $*$-isomorphic to $C(\Sigma_\A)$. In particular, we see that $C_{e,r}^*(\A,\sigma)$ is $*$-isomorphic to $C(\Sigma_\A)$.
\end{proof}

\bibliographystyle{plain}

\bibliography{/Users/raphaelclouatre/Dropbox/Research/Shared/Chris-Raphael/biblio_main_cbshilov}

\def\cprime{$'$}
\begin{bibdiv}
\begin{biblist}

\bib{agler1988}{incollection}{
      author={Agler, Jim},
       title={An abstract approach to model theory},
        date={1988},
   booktitle={Surveys of some recent results in operator theory, {V}ol.\ {II}},
      series={Pitman Res. Notes Math. Ser.},
      volume={192},
   publisher={Longman Sci. Tech., Harlow},
       pages={1\ndash 23},
      review={\MR{976842 (90a:47016)}},
}

\bib{agler2002}{book}{
      author={Agler, Jim},
      author={McCarthy, John~E.},
       title={Pick interpolation and {H}ilbert function spaces},
      series={Graduate Studies in Mathematics},
   publisher={American Mathematical Society, Providence, RI},
        date={2002},
      volume={44},
        ISBN={0-8218-2898-3},
         url={http://dx.doi.org.proxy.lib.uwaterloo.ca/10.1090/gsm/044},
      review={\MR{1882259 (2003b:47001)}},
}

\bib{arvesonUEP}{unpublished}{
      author={Arveson, William},
       title={Notes on the unique extension property},
        note={unpublished notes available at
  https://math.berkeley.edu/$\sim$arveson/Dvi/unExt.pdf},
}

\bib{arveson1969}{article}{
      author={Arveson, William},
       title={Subalgebras of {$C^{\ast} $}-algebras},
        date={1969},
        ISSN={0001-5962},
     journal={Acta Math.},
      volume={123},
       pages={141\ndash 224},
      review={\MR{0253059 (40 \#6274)}},
}

\bib{arveson1972}{article}{
      author={Arveson, William},
       title={Subalgebras of {$C^{\ast} $}-algebras. {II}},
        date={1972},
        ISSN={0001-5962},
     journal={Acta Math.},
      volume={128},
      number={3-4},
       pages={271\ndash 308},
      review={\MR{0394232 (52 \#15035)}},
}

\bib{arveson1998}{article}{
      author={Arveson, William},
       title={Subalgebras of {$C^*$}-algebras. {III}. {M}ultivariable operator
  theory},
        date={1998},
        ISSN={0001-5962},
     journal={Acta Math.},
      volume={181},
      number={2},
       pages={159\ndash 228},
         url={http://dx.doi.org/10.1007/BF02392585},
      review={\MR{1668582 (2000e:47013)}},
}

\bib{arveson2008}{article}{
      author={Arveson, William},
       title={The noncommutative {C}hoquet boundary},
        date={2008},
        ISSN={0894-0347},
     journal={J. Amer. Math. Soc.},
      volume={21},
      number={4},
       pages={1065\ndash 1084},
         url={http://dx.doi.org/10.1090/S0894-0347-07-00570-X},
      review={\MR{2425180 (2009g:46108)}},
}

\bib{arveson2011}{article}{
      author={Arveson, William},
       title={The noncommutative {C}hoquet boundary {II}: hyperrigidity},
        date={2011},
        ISSN={0021-2172},
     journal={Israel J. Math.},
      volume={184},
       pages={349\ndash 385},
         url={http://dx.doi.org/10.1007/s11856-011-0071-z},
      review={\MR{2823981}},
}

\bib{blecher2001}{article}{
      author={Blecher, David~P.},
       title={The {S}hilov boundary of an operator space and the
  characterization theorems},
        date={2001},
        ISSN={0022-1236},
     journal={J. Funct. Anal.},
      volume={182},
      number={2},
       pages={280\ndash 343},
         url={http://dx.doi.org.uml.idm.oclc.org/10.1006/jfan.2000.3734},
      review={\MR{1828796}},
}

\bib{blecherlemerdy2004}{book}{
      author={Blecher, David~P.},
      author={Le~Merdy, Christian},
       title={Operator algebras and their modules---an operator space
  approach},
      series={London Mathematical Society Monographs. New Series},
   publisher={The Clarendon Press, Oxford University Press, Oxford},
        date={2004},
      volume={30},
        ISBN={0-19-852659-8},
  url={http://dx.doi.org.uml.idm.oclc.org/10.1093/acprof:oso/9780198526599.001.0001},
        note={Oxford Science Publications},
      review={\MR{2111973}},
}

\bib{BRS1990}{article}{
      author={Blecher, David~P.},
      author={Ruan, Zhong-Jin},
      author={Sinclair, Allan~M.},
       title={A characterization of operator algebras},
        date={1990},
        ISSN={0022-1236},
     journal={J. Funct. Anal.},
      volume={89},
      number={1},
       pages={188\ndash 201},
         url={http://dx.doi.org.uml.idm.oclc.org/10.1016/0022-1236(90)90010-I},
      review={\MR{1040962}},
}

\bib{brownozawa2008}{book}{
      author={Brown, Nathanial~P.},
      author={Ozawa, Narutaka},
       title={{$C^*$}-algebras and finite-dimensional approximations},
      series={Graduate Studies in Mathematics},
   publisher={American Mathematical Society, Providence, RI},
        date={2008},
      volume={88},
        ISBN={978-0-8218-4381-9; 0-8218-4381-8},
         url={http://dx.doi.org.uml.idm.oclc.org/10.1090/gsm/088},
      review={\MR{2391387 (2009h:46101)}},
}

\bib{choi1977}{article}{
      author={Choi, Man~Duen},
      author={Effros, Edward~G.},
       title={Injectivity and operator spaces},
        date={1977},
     journal={J. Functional Analysis},
      volume={24},
      number={2},
       pages={156\ndash 209},
      review={\MR{0430809}},
}

\bib{CD2016abscont}{article}{
      author={Clou\^atre, Rapha\"el},
      author={Davidson, Kenneth~R.},
       title={Absolute continuity for commuting row contractions},
        date={2016},
        ISSN={0022-1236},
     journal={J. Funct. Anal.},
      volume={271},
      number={3},
       pages={620\ndash 641},
         url={http://dx.doi.org.uml.idm.oclc.org/10.1016/j.jfa.2016.04.012},
      review={\MR{3506960}},
}

\bib{CD2016duality}{article}{
      author={Clou{\^a}tre, Rapha{\"e}l},
      author={Davidson, Kenneth~R.},
       title={Duality, convexity and peak interpolation in the
  {D}rury-{A}rveson space},
        date={2016},
        ISSN={0001-8708},
     journal={Adv. Math.},
      volume={295},
       pages={90\ndash 149},
         url={http://dx.doi.org.uml.idm.oclc.org/10.1016/j.aim.2016.02.035},
      review={\MR{3488033}},
}

\bib{costea2011}{article}{
      author={Costea, Serban},
      author={Sawyer, Eric~T.},
      author={Wick, Brett~D.},
       title={The {C}orona theorem for the {D}rury-{A}rveson {H}ardy space and
  other holomorphic {B}esov-{S}obolov spaces on the unit ball in
  $\mathbb{C}^n$},
        date={2011},
     journal={Anal. PDE},
      volume={4},
      number={4},
       pages={499\ndash 550},
}

\bib{DK2016}{article}{
      author={Davidson, Kenneth},
      author={Kennedy, Matthew.},
       title={Choquet order and hyperrigidity for function systems},
        date={2016},
     journal={preprint},
         url={arXiv:1608.02334},
}

\bib{davidson2015}{article}{
      author={Davidson, Kenneth~R.},
      author={Kennedy, Matthew},
       title={The {C}hoquet boundary of an operator system},
        date={2015},
        ISSN={0012-7094},
     journal={Duke Math. J.},
      volume={164},
      number={15},
       pages={2989\ndash 3004},
         url={http://dx.doi.org.uml.idm.oclc.org/10.1215/00127094-3165004},
      review={\MR{3430455}},
}

\bib{davidson1998}{article}{
      author={Davidson, Kenneth~R.},
      author={Pitts, David~R.},
       title={Nevanlinna-{P}ick interpolation for non-commutative analytic
  {T}oeplitz algebras},
        date={1998},
        ISSN={0378-620X},
     journal={Integral Equations Operator Theory},
      volume={31},
      number={3},
       pages={321\ndash 337},
         url={http://dx.doi.org.proxy.lib.uwaterloo.ca/10.1007/BF01195123},
      review={\MR{1627901 (2000g:47016)}},
}

\bib{DRS2015}{article}{
      author={Davidson, Kenneth~R.},
      author={Ramsey, Christopher},
      author={Shalit, Orr~Moshe},
       title={Operator algebras for analytic varieties},
        date={2015},
        ISSN={0002-9947},
     journal={Trans. Amer. Math. Soc.},
      volume={367},
      number={2},
       pages={1121\ndash 1150},
  url={http://dx.doi.org.uml.idm.oclc.org/10.1090/S0002-9947-2014-05888-1},
      review={\MR{3280039}},
}

\bib{dixmier1977}{book}{
      author={Dixmier, Jacques},
       title={{$C\sp*$}-algebras},
   publisher={North-Holland Publishing Co., Amsterdam-New York-Oxford},
        date={1977},
        ISBN={0-7204-0762-1},
        note={Translated from the French by Francis Jellett, North-Holland
  Mathematical Library, Vol. 15},
      review={\MR{0458185}},
}

\bib{dritschel2005}{article}{
      author={Dritschel, Michael~A.},
      author={McCullough, Scott~A.},
       title={Boundary representations for families of representations of
  operator algebras and spaces},
        date={2005},
        ISSN={0379-4024},
     journal={J. Operator Theory},
      volume={53},
      number={1},
       pages={159\ndash 167},
      review={\MR{2132691 (2006a:47095)}},
}

\bib{FHL2016}{article}{
      author={Fuller, Adam~H.},
      author={Hartz, Michael},
      author={Lupini, Martino},
       title={Boundary representations of operator spaces, and compact
  rectangular matrix convex sets},
        date={2016},
     journal={preprint arXiv:1610.05828},
}

\bib{gamelin1969}{book}{
      author={Gamelin, Theodore~W.},
       title={Uniform algebras},
   publisher={Prentice-Hall, Inc., Englewood Cliffs, N. J.},
        date={1969},
      review={\MR{0410387}},
}

\bib{hamana1979}{article}{
      author={Hamana, Masamichi},
       title={Injective envelopes of operator systems},
        date={1979},
        ISSN={0034-5318},
     journal={Publ. Res. Inst. Math. Sci.},
      volume={15},
      number={3},
       pages={773\ndash 785},
         url={http://dx.doi.org.uml.idm.oclc.org/10.2977/prims/1195187876},
      review={\MR{566081}},
}

\bib{hamana1999}{article}{
      author={Hamana, Masamichi},
       title={Triple envelopes and \v {S}ilov boundaries of operator spaces},
        date={1999},
        ISSN={0916-6009},
     journal={Math. J. Toyama Univ.},
      volume={22},
       pages={77\ndash 93},
      review={\MR{1744498}},
}

\bib{kleski2014bdry}{article}{
      author={Kleski, Craig},
       title={Boundary representations and pure completely positive maps},
        date={2014},
        ISSN={0379-4024},
     journal={J. Operator Theory},
      volume={71},
      number={1},
       pages={45\ndash 62},
         url={http://dx.doi.org.uml.idm.oclc.org/10.7900/jot.2011oct22.1927},
      review={\MR{3173052}},
}

\bib{kleski2014hyper}{article}{
      author={Kleski, Craig},
       title={Korovkin-type properties for completely positive maps},
        date={2014},
        ISSN={0019-2082},
     journal={Illinois J. Math.},
      volume={58},
      number={4},
       pages={1107\ndash 1116},
         url={http://projecteuclid.org.uml.idm.oclc.org/euclid.ijm/1446819304},
      review={\MR{3421602}},
}

\bib{meyer2001}{article}{
      author={Meyer, Ralf},
       title={Adjoining a unit to an operator algebra},
        date={2001},
        ISSN={0379-4024},
     journal={J. Operator Theory},
      volume={46},
      number={2},
       pages={281\ndash 288},
      review={\MR{1870408}},
}

\bib{muhlysolel1998}{incollection}{
      author={Muhly, Paul~S.},
      author={Solel, Baruch},
       title={An algebraic characterization of boundary representations},
        date={1998},
   booktitle={Nonselfadjoint operator algebras, operator theory, and related
  topics},
      series={Oper. Theory Adv. Appl.},
      volume={104},
   publisher={Birkh\"auser, Basel},
       pages={189\ndash 196},
      review={\MR{1639657}},
}

\bib{paulsen2002}{book}{
      author={Paulsen, Vern},
       title={Completely bounded maps and operator algebras},
      series={Cambridge Studies in Advanced Mathematics},
   publisher={Cambridge University Press},
     address={Cambridge},
        date={2002},
      volume={78},
        ISBN={0-521-81669-6},
      review={\MR{1976867 (2004c:46118)}},
}

\bib{paulsen1984PAMS}{article}{
      author={Paulsen, Vern~I.},
       title={Completely bounded homomorphisms of operator algebras},
        date={1984},
        ISSN={0002-9939},
     journal={Proc. Amer. Math. Soc.},
      volume={92},
      number={2},
       pages={225\ndash 228},
         url={http://dx.doi.org.proxy.lib.uwaterloo.ca/10.2307/2045190},
      review={\MR{754708 (85m:47049)}},
}

\bib{paulsen1984}{article}{
      author={Paulsen, Vern~I.},
       title={Every completely polynomially bounded operator is similar to a
  contraction},
        date={1984},
        ISSN={0022-1236},
     journal={J. Funct. Anal.},
      volume={55},
      number={1},
       pages={1\ndash 17},
         url={http://dx.doi.org/10.1016/0022-1236(84)90014-4},
      review={\MR{733029 (86c:47021)}},
}

\bib{paulsensuen1985}{article}{
      author={Paulsen, Vern~I.},
      author={Suen, Ching~Yun},
       title={Commutant representations of completely bounded maps},
        date={1985},
        ISSN={0379-4024},
     journal={J. Operator Theory},
      volume={13},
      number={1},
       pages={87\ndash 101},
      review={\MR{768304}},
}

\bib{phelps2001}{book}{
      author={Phelps, Robert~R.},
       title={Lectures on {C}hoquet's theorem},
     edition={Second Ed.},
      series={Lecture Notes in Mathematics},
   publisher={Springer-Verlag, Berlin},
        date={2001},
      volume={1757},
        ISBN={3-540-41834-2},
         url={http://dx.doi.org.uml.idm.oclc.org/10.1007/b76887},
      review={\MR{1835574}},
}

\bib{popescu1991}{article}{
      author={Popescu, Gelu},
       title={von {N}eumann inequality for {$(B({\scr H})^n)_1$}},
        date={1991},
        ISSN={0025-5521},
     journal={Math. Scand.},
      volume={68},
      number={2},
       pages={292\ndash 304},
      review={\MR{1129595 (92k:47073)}},
}

\bib{popescu2006}{article}{
      author={Popescu, Gelu},
       title={Operator theory on noncommutative varieties},
        date={2006},
        ISSN={0022-2518},
     journal={Indiana Univ. Math. J.},
      volume={55},
      number={2},
       pages={389\ndash 442},
  url={http://dx.doi.org.proxy.lib.uwaterloo.ca/10.1512/iumj.2006.55.2771},
      review={\MR{2225440 (2007m:47008)}},
}

\bib{ruan1988}{article}{
      author={Ruan, Zhong-Jin},
       title={Subspaces of {$C^*$}-algebras},
        date={1988},
        ISSN={0022-1236},
     journal={J. Funct. Anal.},
      volume={76},
      number={1},
       pages={217\ndash 230},
         url={http://dx.doi.org.uml.idm.oclc.org/10.1016/0022-1236(88)90057-2},
      review={\MR{923053}},
}

\end{biblist}
\end{bibdiv}

\end{document}